\documentclass[11pt,a4paper]{article}
\usepackage{fullpage}
\usepackage[english]{babel}
\usepackage[utf8]{inputenc}
\usepackage[T1]{fontenc}
\usepackage{amsmath}
\usepackage{amsthm}
\usepackage{amssymb}
\usepackage{graphicx}
\usepackage{tikz,pgfplots,pgfplotstable}
\usepackage{color}
\usepackage{hyperref}             
\usepackage{versions}
\usepackage{frcursive}

\newtheorem{definition}{Definition}[section]
\newtheorem{theorem}{Theorem}[section]
\newtheorem{proposition}[theorem]{Proposition}
\newtheorem{corollary}[theorem]{Corollary}
\newtheorem{lemma}[theorem]{Lemma}
\newtheorem{remark}{Remark}[section]
\numberwithin{figure}{section}

\definecolor{red}{rgb}{.64,.19,.22}
\definecolor{green}{rgb}{.19,.64,.22}
\definecolor{blue}{rgb}{.2,.2,.7}
\definecolor{gray}{rgb}{.4,.4,.4}

\DeclareMathOperator{\dist}{dist}

\DeclareMathOperator{\diam}{diam}

\DeclareMathOperator{\Lip}{\Lip}

\let\div\relax
\DeclareMathOperator{\div}{div}

\newcommand{\cC}{{\cal C}}

\newcommand{\N}{\mathbb{N}}
\newcommand{\R}{\mathbb{R}}

\newcommand{\Q}{\mathbb{Q}}

\newcommand{\td}{{\text d}}

\newcommand{\tin}{\text{ in }}
\newcommand{\ton}{\text{ on }}

\newcommand{\lra}{\longrightarrow}
\newcommand{\wt}{\widetilde}

\newcommand{\bs}{\backslash}

\newcommand{\ul}{\underline}
\newcommand{\ol}{\overline}

\newcommand{\imp}{\Rightarrow}
\newcommand{\inj}{\hookrightarrow}
\newcommand{\fint}{\mathop{\int\hspace{-1em}-}}

\newcommand{\ph}{{\varphi}}

\newcommand{\norm}[2]{\left\|{#1}\right\|_{#2}}
\newcommand{\norms}[2]{\|{#1}\|_{#2}}

\newcommand{\fae}{\text{for a.e. }}

\newcommand{\set}[2]{\left\{\left.{#1}\ \right|\ {#2}\right\}}

\title{Generalized Morrey–Campanato estimates for elliptic equations with coefficients of integrable oscillation}
\author{Laurent Seppecher\footnote{Institut Camille Jordan, Ecole Centrale de Lyon \& Université Claude Bernard, Lyon, F-69003, France
(laurent.seppecher@ec-lyon.fr)}}

\begin{document}
\maketitle 

\abstract{This work concerns regularity properties of weak solutions to elliptic equations in divergence form $-\div(a\nabla u) = \div F$, under low regularity assumptions on both the coefficient $a$ and the source term $F$. We introduce generalized Morrey and Campanato spaces extending the classical definitions by replacing uniform boundedness requirements with suitable integrability conditions. Within this framework, we establish regularity estimates for the gradient of weak solutions in these generalized spaces. As applications, we recover classical Hölder and Lebesgue estimates and derive fractional Sobolev regularity results. In particular, the proposed approach yields fractional Sobolev estimates in situations where the coefficient may be discontinuous and the gradient of the solution is not expected to be locally bounded.}


\section{Introduction}

The aim of this article is to propose regularity estimates, using a generalized Morrey--Campanato framework, for weak solutions of elliptic equations in divergence form
\begin{equation}\label{eq:main}
-\div(a\nabla u) = \div F\quad \text{in }\Omega,
\end{equation}
where $\Omega$ is an open bounded domain of $\R^d$ with $d\geq 1$. The proposed approach may be viewed as an extension of the Morrey--Campanato approach to the classical Schauder theory, where Hölder continuity of both $a$ and $F$ is assumed. In the present work, this assumption is relaxed by requiring that $a$ satisfies the following conditions:
\medskip

$(H_1):\ a\in L^\infty(\Omega)$ and $a\geq \ul a$ almost everywhere in $\Omega$ for some positive constant $\ul a$.  
\medskip

$(H_2):$ There exists $g\in  L^q(\Omega)$ with $q\in [1,+\infty]$ and $\alpha\in (0,1]$ such that
\begin{equation}\nonumber
 |a(x)-a(y)|\leq g(x)|x-y|^\alpha \quad \fae\ (x,y)\in \Omega^2.
\end{equation}

The main objective is to obtain gradient estimates for the solution of \eqref{eq:main} under some conditions on $q$ and $\alpha$. The case $q=+\infty$ in $(H_2)$ is the classical $C^{0,\alpha}\big(\ol \Omega\big)$ hypothesis. Remark that this hypothesis is a stronger non-symmetric version of the fractional order Hajłasz condition $|a(x)-a(y)|\leq (g(x)+g(y))|x-y|^\alpha$ that is sometimes used in the literature (see \cite{di2014higher,di2014higher2,baison2017fractional}). The motivation of considering the non-symmetric hypothesis $(H_2)$ is that it can be nicely translated to an oscillation-type property:
\begin{equation}\label{eq:osc}
\omega_\alpha[a]:x\mapsto \sup_{r>0}r^{-\alpha}\norm{a-a_{B_\Omega(x,r)}}{L^\infty(B_\Omega(x,r))}\in L^q(\Omega),
\end{equation}
where $a_{B_\Omega(x,r)}$ is the mean value of the coefficient $a$ in $B(x,r)\cap \Omega$. We refer to subsection \ref{sub:H2} for the discussion on this point.  Re: programmation soutenance TFE

In this paper, we introduce generalized Morrey--Campanato spaces extending the classical definitions (see \cite{rafeiro2012morrey}). These spaces allow us to establish regularity estimates for $\nabla u$ under assumptions $(H_1)$ and $(H_2)$. As applications, we recover classical Hölder and Lebesgue estimates and obtain fractional Sobolev regularity results. 

\paragraph{Context:} The regularity theory for weak solutions of elliptic equation is a vast subject in the literature starting from the Hölder regularity results attributed to de Giorgi, Nash and Moser. See \cite{de1957sulla,nash1958continuity,moser1960new} for their foundational works. 

A major topic in this field is to understand the conditions under which the regularity (or integrability) of the source term $F$ is transmitted to the gradient of the weak solution $u$. The classical Schauder estimate provides that $\nabla u$ is $\cC^{0,\alpha}$ as soon as $a$ and $F$ are also $\cC^{0,\alpha}$. One of the standard approaches relies on Morrey--Campanato spaces, (see \cite{rafeiro2012morrey} for an overview on these spaces). The main idea to prove such estimate is to show that, under the condition $a\in \cC^{0,\alpha}$, a Morrey--type and a Campanato--type regularity are transmitted form $F$ to $\nabla u$. Then using equivalences between Campanato spaces and Hölder spaces, the Hölder regularity is deduced. This estimate can be found in many monographs such as \cite{gilbarg1998elliptic,giaquinta2013introduction,troianiello2013elliptic} for instance. 

Another major direction concerns $L^p$ estimates for the gradient.
The seminal result of Meyers \cite{meyers1963p} shows that
$\nabla u\in L^p$ for $p$ slightly greater than $2$ with no smoothness hypothesis on the coefficient. Under additional regularity assumptions,
Calderón--Zygmund type estimates yield $F\in L^p \imp \nabla u\in L^p$
and, in some situations, fractional Sobolev estimates \cite{mingione2007calderon}. 

More recent works have investigated fractional order Sobolev regularity of the solution in term the regularity of both the coefficient $a$  and the source term $F$. It turns out that different regimes exist for estimates in $W^{\alpha,p}$, in the three cases $\alpha q>d$, $\alpha q=d$ and $\alpha q<d$. See \cite{baison2017beltrami,clop2009beltrami,cruz2013beltrami,kuusi2012universal}. The main raison is that in the case $\alpha q=d$ and $\alpha q<d$, we loose the global boundedness of $\nabla u$. 

Other recent works \cite{kristensen2010boundary,baison2017fractional} have investigated low regularity hypothesis on the coefficient $a$ of the form 
\begin{equation}\label{eq:sym}
 |a(x)-a(y)|\leq (g(x)+g(y))|x-y|^\alpha \quad \fae\ (x,y)\in \Omega^2.
\end{equation}
(and some variants) where $g\in L^q(\Omega)$. Various integrability results are obtained using regularity estimates using Besov spaces. 

The approach presented in this paper differs from the mentioned work as the hypothesis $(H_2)$ is slightly stronger than \eqref{eq:sym}. Such hypothesis can be used under the form \eqref{eq:osc} in the proof of the Schauder estimate using an extended definition of the Morrey--Campanato spaces. Since $\omega_\alpha[a]$ is only assumed to belong to $L^q(\Omega)$, the classical Morrey--Campanato framework cannot be applied directly. This motivates the introduction of generalized Morrey and Campanato spaces in which the usual uniform control is replaced by an $L^q$-integrability condition. Within this framework, an extension of the Campanato approach to Schauder estimates can be carried out.

The generalized Morrey and Campanato spaces introduced here appear to be
related in spirit to several generalized Morrey-type constructions considered
in the literature. However, we did not find any direct characterization,
equivalence, or embedding results that would allow the present theory to be
deduced from existing frameworks. 

From the proposed generalized
Morrey--Campanato estimates, we recover both classical and less classical regularity results
and derive some new fractional Sobolev estimates. In particular, we obtain fractional Sobolev
estimates in the regime $\alpha q>d$ for low regularity source terms, as well as new fractional Sobolev estimates in the regime $\alpha q\le d$ where the coefficient $a$ may be discontinuous and the gradient is not expected to be locally bounded.
To the best of our knowledge, these estimates do not follow from existing results.

\paragraph{Overview of the approach.}
The main contribution of this work is the introduction of a generalized
Morrey--Campanato framework adapted to coefficients satisfying $(H_2)$.
Within this framework, the regularity theory relies on a two-step transfer
mechanism.

The first step consists in propagating a generalized Morrey regularity from the source term
$F$ to the gradient $\nabla u$ under assumption $(H_2)$ with Theorem~\ref{theo:1}. 

The second step upgrades this Morrey regularity into a generalized Campanato regularity.
Theorem~\ref{theo:2} shows that if $\nabla u$ already satisfies a generalized Morrey estimate
and if $F$ enjoys a stronger generalized Campanato regularity, then the same Campanato
regularity is transferred to $\nabla u$. In particular, the Campanato exponent increases
by $2\alpha$, reflecting the regularizing effect of the elliptic equation.

Combining these two transfer mechanisms with the embedding properties of generalized
Morrey--Campanato spaces provides a unified method for various gradient estimates.
Classical Sobolev, Hölder, and $L^p$ regularity results are recovered as particular cases,
while several new estimates are obtained in regimes not covered by the standard theory.

\paragraph{Paper outlines:} In section 2, we propose an extended definition of the Morrey spaces denoted $L^{p,\lambda,q}(\Omega)$. To do so, we demand that the $(p,\lambda)$-concentration function defined by $x\mapsto \sup_{r>0}r^{-\lambda/p}\norm{u}{L^p(B_\Omega(x,r))}$ to belong to $L^q(\Omega)$ instead of being uniformly bounded as it is required for the classical Morrey spaces definition. We then prove a series of injection rules for these generalized Morrey spaces. 

In section 3, we follow the same idea to propose an extended definition of the Campanato spaces denoted $L_C^{p,\lambda,q}(\Omega)$. We also prove a series of injection rules for these generalized Campanato spaces, notably with fractional Sobolev spaces. 

In section 4, we prove the two main results of this work. The first one (Theorem \ref{theo:1}) shows that under the hypotheses $(H_1)$ and $(H_2)$, the weak solution $u$ of \eqref{eq:main} satisfies
\begin{equation}\nonumber
F\in L^{2,\lambda, q'}(\Omega)\quad\imp\quad \nabla u\in L^{2,\lambda, q'}(\Omega'),
\end{equation}
for a smooth subdomain $\Omega'\Subset \Omega$ and under some conditions on $\lambda$ and $q'$. In other words, a generalized Morrey--type regularity is transmitted from $F$ to $\nabla u$. The second main result  (Theorem \ref{theo:2}) shows that if $\nabla u$ has a generalized Morrey--type regularity, one can expect that a  generalized Campanato--type regularity is also transmitted from $F$ to $\nabla u$. More precisely, under the hypotheses $(H_1)$ and $(H_2)$ with exponents $\alpha$ and $q$, we prove that 
\begin{equation}\nonumber
\nabla u\in L^{2,\lambda,q'}(\Omega)\quad  \text{and}\quad F\in L_C^{2,\lambda+2\alpha,q''}(\Omega)\quad\imp\quad \nabla u\in L_C^{2,\lambda+2\alpha,q''}(\Omega'),
\end{equation}
for a smooth subdomain $\Omega'\Subset \Omega$ and under some conditions on $\lambda$ and $q''$. 

In section 5, we show that the use of the two main theorems leads to easily retrieve various known gradient estimates in Hölder spaces, and Lebesgue spaces. We also prove some low partial order Sobolev estimates under weak conditions on the elliptic coefficient  $a$ which are new, to the best of our knowledge. In particular, in Corollary \ref{cor:low}, we deal with the case $\alpha q\leq d$ in the hypothesis $(H_2)$ which allows for a discontinuous elliptic coefficient $a$.

\section{Generalized Morrey spaces}

We recall that throughout this work, $\Omega$ is a bounded open domain of $\R^d$. In some results, we will need the following smoothness hypothesis for the domain:  

\begin{definition} For all $x\in\Omega$ and $r>0$, denote $B_\Omega(x,r):=B(x,r)\cap\Omega$ and $|B_\Omega(x,r)|$ its Lebesgue measure. For $\theta>0$, we say that $\Omega$ is $\theta$-smooth if
\begin{equation}\nonumber
\forall x\in\Omega,\quad \forall r\in (0,\diam\Omega),\quad |B_\Omega(x,r)|\geq \theta v_dr^d.
\end{equation}
where $v_d$ is the volume of the unit ball in $\R^d$. For example, a Lipschitz domain is $\theta$-smooth.
\end{definition}

 Classical Morrey spaces $L^{p,\lambda}(\Omega)$  are defined as classes of functions in $L^p(\Omega)$ such that 
\begin{equation}\nonumber
\sup_{x\in\Omega}\sup_{r>0}r^{-\lambda}\int_{B_\Omega(x,r)}|u|^p<+\infty. 
\end{equation}
For background on Morrey and Campanato spaces, see \cite{rafeiro2012morrey}. A natural way to relax the classical Morrey condition is to replace the uniform control with respect to the center variable $x$ by an
$L^q$-integrability condition on $\Omega$. This leads to the following generalized Morrey spaces definition.

\begin{definition} For any $p\in[1,+\infty)$, $\lambda\geq 0$, $u\in L^p(\Omega)$, we define the $(p,\lambda)$-concentration function of $u$ as 
\begin{equation}
E_\Omega^{p,\lambda}[u](x):=\sup_{r>0}r^{-\lambda/p}\norm{u}{L^p(B_\Omega(x,r))},\quad \forall x\in\Omega.
\end{equation}
For any $q\in [1,+\infty]$, we define the generalized Morrey space as the class of all functions in $L^p(\Omega)$ such that $E_\Omega^{p,\lambda}[u]$ belongs to $L^q(\Omega)$:
\begin{equation}\nonumber
L^{p,\lambda,q}(\Omega):=\set{u\in L^p(\Omega)}{E_\Omega^{p,\lambda}[u]\in L^q(\Omega)}. 
\end{equation}
endowed with the norm
\begin{equation}\nonumber
\norm{u}{L^{p,\lambda,q}(\Omega)}:=\norms{E_\Omega^{p,\lambda}[u]}{L^{q}(\Omega)}. 
\end{equation}
\end{definition}

The classical Morrey spaces are recovered as the particular case
$L^{p,\lambda,\infty}(\Omega)$. We begin by collecting some basic
properties of the concentration function $E_\Omega^{p,\lambda}[u]$,
which will be useful for establishing the embedding results below.
These properties are classical from the standard Morrey--Campanato spaces theory. 

\begin{proposition}\label{prop:E} For any $p\in[1,+\infty)$, $\lambda\geq 0$ and $u\in L^p(\Omega)$,

\begin{enumerate}

\item  For all $x\in \Omega$, $E_\Omega^{p,0}[u](x)=\norm{u}{L^p(\Omega)}$. 

\item For all $x\in \Omega$, $E_\Omega^{p,\lambda}[u](x)\geq \diam(\Omega)^{-\lambda/p}\norm{u}{L^p(\Omega)}$. 

\item For all $x\in\Omega$, the map $u\mapsto E_\Omega^{p,\lambda}[u](x)$ is a norm (possibly infinite) in $L^p(\Omega)$. 

\item For all $p'\geq p$, and $\lambda'\geq 0$ such that $(d-\lambda')/p'\leq(d-\lambda)/p$, we have $$E_\Omega^{p,\lambda}[u]\leq C E_\Omega^{p',\lambda'}[u],$$ where $C:=C(d,p,p',\lambda,\lambda',\Omega)>0$. 

\item For almost every $x\in\Omega$, $E_\Omega^{p,d}[u](x)\geq v_d^{1/p}|u|(x)$. 

\item For all $p'\geq p$ and $\lambda\leq d(1-p/p')$, we have $$E_\Omega^{p,\lambda}[u]\leq C\norm{u}{L^{p'}(\Omega)},$$ where  $C:=C(d,p,p',\lambda,\Omega)>0$.

\item For $\lambda\leq d$, then $E_\Omega^{p,\lambda}[u]\leq C\norm{u}{L^\infty(\Omega)}$ where $C:=C(d,p,\lambda,\Omega)>0$. 

\end{enumerate}

\end{proposition}

\begin{proof} 1., 2. and 3. are trivial from the definition. We recall that $v_d:=|B(0,1)|$ is the measure of the unit ball in $\R^d$. \smallskip

\noindent 4. Call $B:=B_\Omega(x,r)$, by Hölder inequality, we have $\norm{u}{L^{p}(B)}\leq \norm{u}{L^{p'}(B)}v_d^{1/p^*}r^{d/p^*}$ where $1/p^*=1/p-1/p'$. Hence,
\begin{equation}\nonumber
r^{-\lambda/p}\norm{u}{L^{p}(B)}\leq  v_d^{1/p^*}r^{-\lambda'/p'}\norm{u}{L^{p'}(B)}r^\delta,
\end{equation}
with $\delta:=(d-\lambda)/p-(d-\lambda')/p'\geq 0$. For $r\leq \diam(\Omega)$, we get
\begin{equation}\nonumber
r^{-\lambda/p}\norm{u}{L^{p}(B)}\leq  v_d^{1/p^*}r^{-\lambda'/p'}\norm{u}{L^{p'}(B)}\diam(\Omega)^\delta
\end{equation}
and for $r> \diam(\Omega)$ we have using 2.,
\begin{equation}\nonumber
\begin{aligned}
r^{-\lambda/p}\norm{u}{L^{p}(B)} &\leq  \diam(\Omega)^{-\lambda/p}\norm{u}{L^p(\Omega)}\leq \diam(\Omega)^{-\lambda/p}|\Omega|^{1/p^*}\norm{u}{L^{p'}(\Omega)}\\
&\leq  \diam(\Omega)^{\lambda'/p'-\lambda/p}|\Omega|^{1/p^*}E_\Omega^{p',\lambda'}[u](x). 
\end{aligned}
\end{equation}
We conclude by taking the supremum over $r>0$. \smallskip

\noindent 5. The Lebesgue differentiation theorem provides that

\begin{equation}\nonumber
\lim_{r\to 0}v_d^{-1}r^{-d}\int_{B_\Omega(x,r)}|u|^p=|u|^p(x),\quad \text{f.a.e. } x\in \Omega.
\end{equation}
As a consequence, $E_\Omega^{p,d}[u](x)\geq \lim_{r\to 0}r^{-d/p}\norm{u}{L^p(B_\Omega(x,r))}=v_d^{1/p}|u|(x)$ for almost every $x\in\Omega$.\smallskip

\noindent 6. Apply point 4 with $\lambda'=0$ then apply point 1.\smallskip

\noindent 7. Call $B:=B_\Omega(x,r)$, by Hölder inequality, we have $\norm{u}{L^{p}(B)}\leq v_d^{1/p}r^{d/p}\norm{u}{L^{\infty}(\Omega)}$. Hence, for $r\leq \diam(\Omega)$, 
\begin{equation}\nonumber
r^{-\lambda/p}\norm{u}{L^{p}(B)}\leq v_d^{1/p}\diam(\Omega)^{(d-\lambda)/p}\norm{u}{L^{\infty}(\Omega)},
\end{equation}
as $d-\lambda\geq 0$. For $r> \diam(\Omega)$, we directly get $r^{-\lambda/p}\norm{u}{L^{p}(B)}\leq \diam(\Omega)^{-\lambda/p}\norm{u}{L^p(\Omega)}\leq C\norm{u}{L^\infty(\Omega)}$. Taking the supremum over $r>0$ gets to the conclusion. 
\end{proof}

\begin{remark} From the above properties, one can check that $\norm{\cdot}{L^{p,\lambda,q}(\Omega)}$ is indeed a norm and makes $L^{p,\lambda,q}(\Omega)$ a Banach space.  
\end{remark}

The previous properties immediately yield several embedding rules for generalized Morrey spaces. Throughout the paper, for two normed spaces $E$ and $F$, the notation $E\inj F$ means that $E$ is continuously embedded into $F$, while $E\sim F$ means that both embeddings $E\inj F$ and $F\inj E$ hold.

\begin{proposition}[Injection rules]\label{prop:injmorrey} For any $p\in[1,+\infty)$, $\lambda\geq 0$ and $q\in [1,+\infty]$ the following holds.  

\begin{enumerate}

\item $L^{p,0,q}(\Omega)\sim L^p(\Omega)$. 

\item For $p'\in [p,+\infty)$, $q'\in [q,+\infty]$ and $\lambda'\geq 0$ such that $(\lambda-d)/p\leq(\lambda'-d)/p'$, we have $$L^{p',\lambda',q'}(\Omega)\inj L^{p,\lambda,q}(\Omega).$$ 

\item For $\lambda>d$, we have $L^{p,\lambda,q}(\Omega)=\{0\}$.

\item $L^{p,d,q}(\Omega)\inj L^q(\Omega)$ and if $p<q\leq +\infty$ and $\Omega$ is $\theta$--smooth, then $L^{p,d,q}(\Omega)\sim L^q(\Omega)$. 

\item If $p\leq q<+\infty$, $L^{q}(\Omega)\inj L^{p,d(1-p/q),\infty}(\Omega)$. 

\end{enumerate}

\end{proposition}

\begin{proof} 1. Direct from Proposition \ref{prop:E} point 1 and applying the $L^q$-norm.\smallskip

\noindent 2. From Proposition \ref{prop:E} point 4, apply the $L^q$-norm. \smallskip

\noindent 3. Take $u\in L^{p,\lambda,q}(\Omega)$. As $u\in L^p(\Omega)$, the Lebesgue differentiation theorem says that 
\begin{equation}\nonumber
\frac{1}{|B_\Omega(x,r)|}\int_{B_\Omega(x,r)}|u|^p(y)\td y \overset{r\to 0}\lra |u|^p(x),\quad \fae x\in\Omega. 
\end{equation}
Fix $x\in\Omega$ such that the above convergence is true and so that $E_\Omega^{p,\lambda}[u](x)<+\infty$ (true for almost every $x\in\Omega$).  We remark that for $r$ small enough, $B(x,r)\subset \Omega$ and 
\begin{equation}\nonumber
\frac{1}{|B_\Omega(x,r)|}\int_{B_\Omega(x,r)}|u|^p \leq \frac{r^{\lambda-d}}{v_d} \left(E_\Omega^{p,\lambda}[u](x)\right)^p \overset{r\to 0}\lra 0,
\end{equation}
as $\lambda>d$. As a consequence, $u(x)=0$ for almost every $x$ in $\Omega$. 
\smallskip

\noindent 4. From Proposition \ref{prop:E} point 5 we get 
\begin{equation}\nonumber
|u|\leq  v_d^{-1/p}E_\Omega^{p,d}[u]\quad\text{a.e. in }\Omega.
\end{equation}
Taking the $L^q$-norm, we get the first injection. For the reciprocal injection, take $q\in(p,+\infty)$ and define the maximal function of Hardy–Littlewood 
\begin{equation}\nonumber
\forall v\in L^r(\Omega),\quad M[v](x):=\sup_{r>0}\frac{1}{|B_\Omega(x,r)|}\int_{B_\Omega(x,r)}|v|,\quad \forall x\in\Omega.
\end{equation}
It is known that the operator $M$ is bounded from $L^s(\Omega)$ to  $L^s(\Omega)$ as soon as $1<s\leq+\infty$ and $\Omega$ being $\theta$--smooth, i.e. there exists $C>0$ such that
\begin{equation}\nonumber
\forall v\in L^r(\Omega),\quad \norm{M[v]}{L^s(\Omega)}\leq C \norm{v}{L^s(\Omega)}. 
\end{equation}
See \cite{stein1970singular}. We then remark that $E_\Omega^{p,d}[u]^p \leq 1/(\theta v_d) M[u^p]$ and then as $q>p$, 
\begin{equation}\nonumber
\begin{aligned}
E_\Omega^{p,d}[u]^q &\leq \wt C M[u^p]^{q/p}\\
\int_\Omega E_\Omega^{p,d}[u]^q &= \wt C\int_\Omega M[u^p]^{q/p} = \wt C\norm{M[u^p]}{L^{q/p}(\Omega)}^{q/p} \leq \wt CC^{q/p}\norm{u^p}{L^{q/p}(\Omega)}^{q/p} = \wt CC^{q/p}\int_\Omega |u|^q\\
\norm{E_\Omega^{p,d}[u]}{L^q(\Omega)} &\leq \wt C^{1/q}C^{1/p}\norm{u}{L^q(\Omega)},
\end{aligned}
\end{equation}
This proves the reciprocal injection. The case $q=+\infty$ is directly deduced from Proposition \ref{prop:E} point 7.\smallskip

\noindent 5. From Proposition \ref{prop:E} point 6, we just apply the $L^\infty$--norm.  
\end{proof}

\begin{remark}
The parameter $\lambda=d$ plays a distinguished role.
In this case, the generalized Morrey spaces coincide with
Lebesgue spaces:
\[
L^{p,d,q}(\Omega)\sim L^q(\Omega),
\]
when $\Omega$ is $\theta$--smooth. 
\end{remark}

\section{Generalized Campanato spaces}

In this section, we extend the classical Campanato spaces following the same philosophy as for generalized Morrey spaces. We then establish several embedding results extending the classical Campanato theory. In particular, we prove links between generalized Campanato spaces and fractional Sobolev spaces. For any measurable set $D\subset\Omega$ with positive measure and any $u\in L^1(\Omega)$, we denote its average on $D$ by
\begin{equation}\nonumber
u_D:=\fint_D u := \frac{1}{|D|}\int_D u.
\end{equation}
For any $x\in\Omega$ and $r>0$, we define
\begin{equation}\nonumber
B_\Omega(x,r):=B(x,r)\cap\Omega.
\end{equation}
We also denote by $u_r(x)$ the average of $u$ on $B_\Omega(x,r)$:
\begin{equation}\nonumber
u_r(x):=u_{B_\Omega(x,r)}
=\frac{1}{|B_\Omega(x,r)|}\int_{B_\Omega(x,r)}u.
\end{equation}

\begin{definition} For any $p\in[1,+\infty)$, $\lambda\geq 0$, $u\in L^p(\Omega)$, we define the $(p,\lambda)$--oscillation function of $u$ by 
\begin{equation}
O_\Omega^{p,\lambda}[u](x):=\sup_{r>0}r^{-\lambda/p}\norm{u-u_r(x)}{L^p(B_\Omega(x,r))},\quad \forall x\in\Omega.
\end{equation}
For any $q\in [1,+\infty]$, we define the generalized Campanato space as the class of all functions in $L^p(\Omega)$ such that $O_\Omega^{p,\lambda}[u]$ belongs to $L^q(\Omega)$:
\begin{equation}\nonumber
L_C^{p,\lambda,q}(\Omega):=\set{u\in L^p(\Omega)}{O_\Omega^{p,\lambda}[u]\in L^q(\Omega)}. 
\end{equation}
endowed with the norm
\begin{equation}\nonumber
\norm{u}{L_C^{p,\lambda,q}(\Omega)}:=\norm{u}{L^{p}(\Omega)}+\norms{O_\Omega^{p,\lambda}[u]}{L^{q}(\Omega)}. 
\end{equation}
\end{definition}

\begin{remark}
It is straightforward to verify that
$\norm{\cdot}{L_C^{p,\lambda,q}(\Omega)}$
defines a norm on
$L_C^{p,\lambda,q}(\Omega)$.
Moreover, equipped with this norm,
$L_C^{p,\lambda,q}(\Omega)$ is a Banach space.
\end{remark}

\begin{lemma} We can use an alternative, and very useful, equivalent definition of the $(p,\lambda)$--oscillation:
\begin{equation}
\wt O_\Omega^{p,\lambda}[u](x):=\sup_{r>0}r^{-\lambda/p}\inf_{c\in\R}\norm{u-c}{L^p(B_\Omega(x,r))}
\end{equation}
as we have $$\wt O_\Omega^{p,\lambda}[u]\leq O_\Omega^{p,\lambda}[u]\leq 2\wt O_\Omega^{p,\lambda}[u].$$
\end{lemma}

\begin{proof} The first inequality is obvious. For the second one, fix $c\in\R$, $x\in\Omega$ and $r>0$. By the triangle
inequality we write
\begin{equation}\nonumber
\norm{u-u_r(x)}{L^p(B_\Omega(x,r))}
\leq
\norm{u-c}{L^p(B_\Omega(x,r))}
+
\norm{u_r(x)-c}{L^p(B_\Omega(x,r))}.
\end{equation}
Moreover,
\begin{equation}\nonumber
|u_r(x)-c|
\leq
\fint_{B_\Omega(x,r)} |u-c|
\leq
|B_\Omega(x,r)|^{-1/p}
\norm{u-c}{L^p(B_\Omega(x,r))},
\end{equation}
and therefore
\begin{equation}\nonumber
\norm{u_r(x)-c}{L^p(B_\Omega(x,r))}
\leq
\norm{u-c}{L^p(B_\Omega(x,r))}.
\end{equation}
Thus,
\begin{equation}\nonumber
\norm{u-u_r(x)}{L^p(B_\Omega(x,r))}
\leq
2\norm{u-c}{L^p(B_\Omega(x,r))}.
\end{equation}
Taking the infimum over $c\in\R$ and then the supremum over $r>0$ gives the demanded inequality
$O_\Omega^{p,\lambda}[u]\leq2\wt O_\Omega^{p,\lambda}[u]$.
\end{proof}

We now prove a series of embedding results for generalized Campanato spaces. Some of them follow directly from the definitions, whereas others require technical results concerning the relationship between the $(p,\lambda)$--oscillation function and the local mean function. These results are presented in Appendix \ref{app:localmean}. It may therefore be useful to consult this appendix before reading the proofs of the following embeddings.

\begin{proposition}[Injection rules]\label{prop:injcamp}\ Let $p\in [1,+\infty)$, $\lambda\geq 0$ and $q\in [1,+\infty]$,

\begin{enumerate}

\item $L^{p,\lambda,q}(\Omega)\inj L_C^{p,\lambda,q}(\Omega)$. 

\item For $p'\in [p,+\infty)$, $q'\in [q,+\infty]$ and $\lambda'\geq 0$ such that $(\lambda-d)/p\leq(\lambda'-d)/p'$, we have  $$L_C^{p',\lambda',q'}(\Omega)\inj L_C^{p,\lambda,q}(\Omega).$$ 

\item If $\Omega$ is $\theta-smooth$, for $\lambda\in [0,d)$, $L_C^{p,\lambda,q}(\Omega)\sim L^{p,\lambda,q}(\Omega)$. 

\item  If $\Omega$ is $\theta-smooth$, for $0\leq \alpha'<\alpha<1$ and $q<+\infty$ we have $L_C^{p,d+\alpha p,q}(\Omega)\inj W^{\alpha',q}(\Omega)$.  

\item If $\Omega$ is $\theta-smooth$, for $\lambda>d$, $L_C^{p,\lambda,q}(\Omega)\inj L^{p,d,q}(\Omega)$. 

\item For $\alpha\in(0,1)$, if $p\leq q$ then $W^{\alpha,q}(\Omega)\inj L_C^{p,d+\alpha p,q}(\Omega)$. 

\item If $\Omega$ is $\theta-smooth$ and $p< q<+\infty$ then $L_C^{p,d,q}(\Omega)\sim L^q(\Omega)$. 

\item If $\Omega$ is $\theta-smooth$, for $\alpha\in(0,1)$ we have $L_C^{p,d+\alpha p,\infty}(\Omega)\sim \cC^{0,\alpha}\big(\ol\Omega\big)$. 
\end{enumerate}
\end{proposition}

\begin{proof} 1. It directly comes from $\wt O_\Omega^{p,\lambda}[u]\leq E_\Omega^{p,\lambda}[u]$. \smallskip

\noindent 2. Take $u\in L_C^{p',\lambda',q'}(\Omega)$. First remark that $u\in L^{p'}(\Omega)\inj L^p(\Omega)$ as $p'\geq p$.  Call $B=B_\Omega(x,r)$, and $s:=(1/p-1/p')^{-1}$ if $p'>p$ or $s=+\infty$ if $p'=p$, for  $r\leq r_0:=\diam(\Omega)$ we have 
\begin{equation}\nonumber
\begin{aligned}
\norm{u-u_r(x)}{L^{p}(B)} &\leq \norm{u-u_r(x)}{L^{p'}(B)}v_d^{1/{s}}r^{d/{s}}\\
r^{-\lambda/p}\norm{u-u_r(x)}{L^{p}(B)} &\leq v_d^{1/s} r^{d/s-\lambda/p}\norm{u-u_r(x)}{L^{p'}(B)}\leq v_d^{1/s} r^\delta r^{-\lambda'/p'}\norm{u-u_r(x)}{L^{p'}(B)}\\
&\leq v_d^{1/s} r_0^\delta O_\Omega^{p',\lambda'}[u](x)
\end{aligned}
\end{equation}
where $\delta = (\lambda'-d)/p'- (\lambda-d)/p\geq 0$. For $r> r_0$ we have 
\begin{equation}\nonumber
\begin{aligned}
r^{-\lambda/p}\norm{u-u_r(x)}{L^{p}(B)} &\leq r_0^{-\lambda/p}\norm{u-u_\Omega}{L^{p}(\Omega)}\\
&\leq 2r_0^{-\lambda/p}\norm{u}{L^{p}(\Omega)}\leq C_1\norm{u}{L^{p'}(\Omega)}.  
\end{aligned}
\end{equation}
Taking the supremum over $r>0$, then the $L^q$-norm over $\Omega$, and using the embedding
$L^{q'}(\Omega)\inj L^q(\Omega)$, we obtain
\[
\norm{u}{L_C^{p,\lambda,q}(\Omega)}
\leq C_3\norm{u}{L_C^{p',\lambda',q'}(\Omega)}.
\]

\noindent 3. Take $u\in L_C^{p,\lambda,q}(\Omega)$ and call $B:=B_\Omega(x,r)$. For any $r\geq r_0:=\diam(\Omega)$, we have 
\begin{equation}\nonumber
r^{-\lambda/p}\norm{u}{L^p(B)}\leq r_0^{-\lambda/p}\norm{u}{L^p(\Omega)}. 
\end{equation}
For $r\leq r_0$, 
\begin{equation}\nonumber
\begin{aligned}
\norm{u}{L^p(B)} &\leq \norm{u-u_r(x)}{L^p(B)} + \norm{u_r(x)}{L^p(B)}\\
&\leq r^{\lambda/p}O_\Omega^{p,\lambda}[u](x) + |B|^{1/p}|u_r(x)|\\
&\leq r^{\lambda/p}O_\Omega^{p,\lambda}[u](x) + v_d^{1/p}r^{d/p}|u_r(x)|\\
r^{-\lambda/p}\norm{u}{L^p(B)} &\leq O_\Omega^{p,\lambda}[u](x) + v_d^{1/p}r^{(d-\lambda)/p}|u_r(x)|\\
\end{aligned}
\end{equation}
Using Proposition \ref{prop:bound} we have 
\begin{equation}\nonumber
|u_r(x)|\leq \frac{1}{|\Omega|^{1/p}}\norm{u}{L^p(\Omega)} + C_4r^{(\lambda-d)/p}O_\Omega^{p,\lambda}[u](x)
\end{equation}
which leads to
\begin{equation}\nonumber
r^{-\lambda/p}\norm{u}{L^p(B)} \leq (1+v_d^{1/p}C_4)O_\Omega^{p,\lambda}[u](x) + \frac{v_d^{1/p}r_0^{(d-\lambda)/p}}{|\Omega|^{1/p}}\norm{u}{L^p(\Omega)},
\end{equation}
Combining the estimates for $r\le r_0$ and $r>r_0$, we obtain
\[
E_\Omega^{p,\lambda}[u](x)
\le
C\left(
O_\Omega^{p,\lambda}[u](x)
+
\|u\|_{L^p(\Omega)}
\right).
\]
Now take the $L^q$-norm to conclude.\smallskip

\noindent 4. From point 2, we have $L_C^{p,d+\alpha p,q}(\Omega)\inj L_C^{1,d+\alpha,q}(\Omega)$. Apply Corollary \ref{cor:39} to write 
\begin{equation}\begin{aligned}\label{eq:holder1}
|u(x)-u(y)| \leq C_6\left(O_\Omega^{1,d+\alpha}[u](x)+O_\Omega^{1,d+\alpha}[u](y)\right)|x-y|^{\alpha},\quad \fae\ (x,y)\in\Omega^2.
\end{aligned}\end{equation}
Call $G:=O_\Omega^{1,d+\alpha}[u]\in L^q(\Omega)$, we have for $\alpha'<\alpha$
\begin{equation}\nonumber
\begin{aligned}
\frac{|u(x)-u(y)|^q}{|x-y|^{d+q\alpha'}}\leq C\left(\frac{G(x)^q}{|x-y|^{d-q(\alpha-\alpha')}}+\frac{G(y)^q}{|x-y|^{d-q(\alpha-\alpha')}}\right). 
\end{aligned}
\end{equation}
We now integrate in $x$ and $y$ on $\Omega$ (the left-hand term is integrable as $\alpha-\alpha'>0$) to get that $u\in W^{\alpha',q}(\Omega)$. \smallskip

\noindent 5. Take $u\in L_C^{p,\lambda,q}(\Omega)$ and call $B:=B_\Omega(x,r)$. For any $r\leq r_0:=\diam(\Omega)$, we have 
\begin{equation}\nonumber
\begin{aligned}
\norm{u}{L^p(B)} &\leq \norm{u-u_r(x)}{L^p(B)} + \norm{u_r(x)}{L^p(B)}\\
&\leq r^{\lambda/p}O_\Omega^{p,\lambda}[u](x) + v_d^{1/p}r^{d/p}|u_r(x)|\\
r^{-d/p}\norm{u}{L^p(B)} &\leq r_0^{(\lambda-d)/p}O_\Omega^{p,\lambda}[u](x) + v_d^{1/p}|u_r(x)|.\\
\end{aligned}
\end{equation}
From the Proposition \ref{prop:36} we know that $u_r(x)$ is bounded by 
\begin{equation}\nonumber
|u_r(x)|\leq |u(x)|+ |u_{r}(x)-u(x)|\leq |u(x)|+Cr^{(\lambda-d)/p} O_\Omega^{p,\lambda}[u](x),\quad\fae x\in\Omega.
\end{equation}
Hence 
\begin{equation}\nonumber
\begin{aligned}
r^{-d/p}\norm{u}{L^p(B)} &\leq C\left(r_0^{(\lambda-d)/p}O_\Omega^{p,\lambda}[u](x) + |u(x)|\right).\\
\end{aligned}
\end{equation}
For $r> r_0$, $r^{-d/p}\norm{u}{L^p(B)}\leq r_0^{-d/p}\norm{u}{L^p(\Omega)}$. Finally, taking the supremum on $r>0$, we get 
\begin{equation}\nonumber
E_\Omega^{p,d}[u]\leq \tilde C(\norm{u}{L^p(\Omega)}+ O_\Omega^{p,\lambda}[u] + |u(x)|),\quad\fae\ x\in\Omega. 
\end{equation}
Taking now the $L^q$-norm of the previous line ($u\in L^q(\Omega)$ from  point 4), we get that $u\in L^{p,d,q}(\Omega)\sim L^q(\Omega)$.\smallskip

\noindent 6. Take $u\in W^{\alpha,q}(\Omega)$, then $u\in L^q(\Omega)$ and 
\begin{equation}\nonumber
G:(x,y)\mapsto \frac{|u(x)-u(y)|}{|x-y|^{\alpha+d/q}}\in L^q(\Omega^2). 
\end{equation}
In $B:=B_\Omega(x,r)$ with $r>0$, we take the $L^p$-norm with respect to $y$ and use Hölder's inequality:
\begin{equation}\nonumber
\begin{aligned}
\norm{u-u(x)}{L^p(B)}
&\leq \norm{G(x,\cdot)}{L^q(B)} |B|^{1/p-1/q} r^{\alpha+d/q} \\
&\leq C \norm{G(x,\cdot)}{L^q(B)} r^{\alpha+d/p}
\end{aligned}
\end{equation}
and then
\begin{equation}\nonumber
\wt O_\Omega^{p,d+\alpha p}[u](x) \leq C\norm{G(x,\cdot)}{L^q(\Omega)},
\end{equation}
where $C$ is a positive constant depending only on $d,p$ and $q$. We conclude by taking the $L^q$-norm over $\Omega$. \smallskip

\noindent 7. From Proposition \ref{prop:injmorrey} point 4 and point 1 in the present proposition, we have
\begin{equation}\nonumber
L^q(\Omega)\sim L^{p,d,q}(\Omega)\inj L_C^{p,d,q}(\Omega).
\end{equation}
Reciprocally, take $u\in L_C^{p,d,q}(\Omega)$, and define the Fefferman--Stein sharp maximal function by
\begin{equation}\nonumber
u^\#(x):=\sup_{r>0}\fint_{B_\Omega(x,r)}
|u(y)-u_r(x)|\,dy,
\quad \forall x\in\Omega.
\end{equation}
The Fefferman--Stein inequality (see for instance \cite{fefferman1972h} or \cite{lerner2010some}) gives for $q\in(1,+\infty)$,
\begin{equation}\nonumber
\norm{u-u_\Omega}{L^q(\Omega)}
\leq C\norms{u^\#}{L^q(\Omega)}
\end{equation}
as soon as $\Omega$ is $\theta$--smooth. Now, by H\"older's inequality, we have
\begin{equation}\nonumber
u^\#(x)\leq O_\Omega^{p,d}[u](x),
\quad \forall x\in\Omega,
\end{equation}
and so
\begin{equation}\nonumber
\norm{u-u_\Omega}{L^q(\Omega)}
\leq C\norms{O_\Omega^{p,d}[u]}{L^q(\Omega)}.
\end{equation}
From that, we write
\begin{equation}\nonumber
\begin{aligned}
\norm{u}{L^q(\Omega)}
&\leq \norm{u-u_\Omega}{L^q(\Omega)}
+\norm{u_\Omega}{L^q(\Omega)} \\
&\leq C\norms{O_\Omega^{p,d}[u]}{L^q(\Omega)}
+|\Omega|^{1/q}|u_\Omega| \\
&\leq C\norms{O_\Omega^{p,d}[u]}{L^q(\Omega)}
+|\Omega|^{1/q-1/p}\norm{u}{L^p(\Omega)} \\
&\leq \wt C(d,p,q,\Omega) \norm{u}{L_C^{p,d,q}(\Omega)}.
\end{aligned}
\end{equation}
This proves the reciprocal injection. \smallskip

\noindent 8. This is the classical Campanato--Hölder equivalence.

\end{proof}

\begin{remark}
The parameter $\lambda=d$ plays a critical role in the generalized Morrey--Campanato scale. If $\lambda\leq d$ and $\Omega$ is $\theta$--smooth, then generalized Morrey and Campanato spaces coincide:
\[
L_C^{p,\lambda,q}(\Omega)\sim L^{p,\lambda,q}(\Omega).
\]
At the critical value $\lambda=d$, both scales are related to Lebesgue spaces. More precisely, for $p<q<+\infty$,
\[
L^{p,d,q}(\Omega)\sim L_C^{p,d,q}(\Omega)\sim L^q(\Omega).
\]
For $\lambda:=>d$, generalized Campanato spaces become strictly larger than the generalized Morrey spaces and are connected to fractional Sobolev and Hölder spaces. The generalized Campanato spaces provide a continuous framework interpolating between Morrey--type, Lebesgue--type and Hölder--type regularity.
\end{remark}

\section{Gradient estimates in generalized Morrey--Campanato spaces}

This section is devoted to the main results of the paper. We first derive several consequences of assumption $(H_2)$ and discuss its relation with local oscillation estimates. We then establish gradient estimates for weak solutions of \eqref{eq:main} in generalized Morrey and Campanato spaces.

\subsection{Properties of the elliptic coefficient  $a$}\label{sub:H2}

Let $a\in L^\infty(\Omega)$ satisfy assumptions $(H_1)$ and $(H_2)$, which we recall below:
\medskip

$(H_1):\ a\in L^\infty(\Omega)$ and $a\geq \ul a>0$ almost everywhere in $\Omega$ for some constant $\ul a$.  
\medskip

$(H_2):$ $\exists g\in  L^q(\Omega)$ with $q\in [1,+\infty]$ and $\alpha\in (0,1]$ such that
\begin{equation}\nonumber
 |a(x)-a(y)|\leq g(x)|x-y|^\alpha \quad \fae\ (x,y)\in \Omega^2.
\end{equation}

The assumption $(H_2)$ is rather unusual in the literature. Nevertheless, it admits a natural interpretation in terms of local oscillations. This observation will play a crucial role in the sequel.

\begin{definition} For any $\alpha\in(0,1]$, the $\alpha$--oscillation of a function $a\in L^\infty(\Omega)$ is defined as
\begin{equation}\nonumber
\omega_{\alpha}[a](x):=\sup_{r>0}r^{-\alpha}\norm{a-a_r(x)}{L^\infty(B_\Omega(x,r))},\quad \forall x\in\Omega.  
\end{equation}
\end{definition}

The next result establishes an equivalence between $(H_2)$ and the fact that $\omega_{\alpha}[a]$ belongs to $L^q(\Omega)$.

\begin{proposition}\label{prop:1} The coefficient $a$ satisfies $(H_2)$ with exponents $\alpha$ and $q$
if and only if
\[
\omega_\alpha[a]\in L^q(\Omega).
\]
Moreover, in this case, $\omega_\alpha[a]\le 2g$ almost everywhere in $\Omega$. 
\end{proposition}

The proof is postponed to Appendix~\ref{app:1}.

\begin{proposition}\label{prop:H2impSobolev} Let $a\in L^\infty(\Omega)$, then 
\begin{equation}\nonumber
a\ \text{satisfies}\ (H_2)\quad\imp\quad a\in W^{\alpha',q}(\Omega),\quad\forall \alpha'<\alpha. 
\end{equation}
In particular, if $\alpha q>d$ and $\Omega$ is $\theta$--smooth, then $a\in\cC^{0,\alpha''}\big(\ol\Omega\big)$ for all $0<\alpha''< \alpha-d/q$. 
\end{proposition}

\begin{proof} For $q<+\infty$, we write
\begin{equation}\nonumber
\begin{aligned}
\frac{|a(x)-a(y)|^q}{|x-y|^{d+\alpha'q}}\leq g^q(x)\frac{1}{|x-y|^{d-(\alpha-\alpha')q}}.
\end{aligned}
\end{equation}
We first integrate this inequality on the variable $y\in\Omega$ by Fubini to get
\begin{equation}\nonumber
\begin{aligned}
\int_{\Omega}\frac{|a(x)-a(y)|^q}{|x-y|^{d+\alpha'q}} \td y\leq g^q(x) \int_{\Omega}\frac{\td y}{|x-y|^{d-(\alpha-\alpha')q}}\leq  g^q(x) \int_{0}^{\diam \Omega}\frac{\td r}{r^{1-(\alpha-\alpha')q}}\leq  Cg^q(x),
\end{aligned}
\end{equation}
where $C$ depends only on $d,\alpha,\alpha'$ and $\Omega$. We now integrate on $x\in\Omega$ to conclude that  $a\in W^{\alpha',q}(\Omega)$. 

 The second statement is due to the Sobolev embedding $W^{\alpha',q}(\Omega)\inj \cC^{0,\alpha''}\big(\ol\Omega\big)$ with $\alpha''< \alpha'-d/q$. The case $q=+\infty$ is direct. 
\end{proof}

\begin{proposition}\label{prop:H2impH2} If $a\in L^\infty(\Omega)$ satisfies $(H_2)$ with exponents $q\in [1,+\infty)$ and $\alpha\in (0,1]$, then it also satisfies $(H_2)$ with exponents $q'\geq q$ and $\alpha'\in (0,1]$ such that $\alpha'q'\leq \alpha q$. 

\end{proposition}

\begin{proof} 
\begin{equation}\nonumber
\begin{aligned}
 |a(x)-a(y)| &\leq  |a(x)-a(y)|^{1-q/q'}|a(x)-a(y)|^{q/q'}\\
 &\leq (2 \norm{a}{L^\infty(\Omega)})^{1-q/q'}g^{q/q'}(x)|x-y|^{\alpha q/q'}\\
 &\leq(2 \norm{a}{L^\infty(\Omega)})^{1-q/q'}g^{q/q'}(x)|x-y|^{\alpha q/q'-\alpha'}|x-y|^{\alpha'}\\
  &\leq (2 \norm{a}{L^\infty(\Omega)})^{1-q/q'}(\diam \Omega)^{\alpha q/q'-\alpha'} g^{q/q'}(x) |x-y|^{\alpha'}\\
  &\leq C g^{q/q'}(x) |x-y|^{\alpha'}\\
\end{aligned}
\end{equation}
 where $C:= (2 \norm{a}{L^\infty(\Omega)})^{1-q/q'}(\diam \Omega)^{\alpha q/q'-\alpha'}$ and $g^{q/q'}\in L^{q'}(\Omega)$. 

\end{proof}

\subsection{Estimate in generalized Morrey spaces}

Let $\Omega'\subset \Omega$ be an open subdomain of $\Omega$. We write
$\Omega'\Subset \Omega$ whenever
$\dist(\Omega',\partial\Omega)>0$. This first gradient estimate concerns the generalized Morrey regularity of
$\nabla u$. Its proof follows the classical Campanato--Morrey approach for
elliptic equations. The main novelty is that the uniform Hölder seminorm of
the coefficient is replaced by the pointwise oscillation function
$\omega_\alpha[a](x)$, whose integrability is controlled by $g\in L^q(\Omega)$. The final
step also differs from the classical argument: instead of relying on a
uniform bound with respect to $x$, we exploit an $L^{q'}$--integrability
estimate for a suitable exponent $q'$.

\begin{theorem}\label{theo:1} Let $u\in H^1(\Omega)$ be a weak solution of \eqref{eq:main} with the coefficient $a$ satisfying the hypotheses $(H_1)$ and $(H_2)$ with $q\in [1,+\infty]$ and $\alpha\in (0,1]$. For $\Omega'\Subset \Omega$, $\lambda\in(0,d)$ with $\lambda\leq 2\alpha q$ and $q'\in [1,+\infty]$ with $q'\leq 2\alpha q/\lambda$, we have
\begin{equation}\nonumber
F\in L^{2,\lambda, q'}(\Omega)\quad\imp\quad \nabla u\in L^{2,\lambda, q'}(\Omega'),
\end{equation}
 and there exists a positive constants $C:=C(d,\alpha,\lambda,q,q',\underline{a},\Omega,\Omega')$ such that
\begin{equation}\nonumber
\norm{\nabla u}{L^{2,\lambda,q'}(\Omega')}\leq C\left(\left(1 + \norm{g}{L^q(\Omega)}^{\lambda/(2\alpha)}\right)\norm{\nabla u}{L^2(\Omega)}+ \norm{F}{L^{2,\lambda,q'}(\Omega)}\right).
 \end{equation}
\end{theorem}

\begin{proof} Consider $\Omega'\Subset\Omega$. From the hypotheses on $a$ and $F$, we know that almost everywhere in $\Omega$ we have, thanks to Proposition \ref{prop:1},  that $\omega_\alpha[a]\leq 2g$ and $g<+\infty$ and also $E_\Omega^{2,\lambda}[F]<+\infty$ as $F\in L^{2,\lambda,q'}(\Omega)$. We then fix $x\in\Omega'$ such that these three inequalities are fulfilled at point $x$.  

We start by defining the function 
\begin{equation}\nonumber
\ph(r):=\int_{B_\Omega(x,r)}|\nabla u|^2,\quad \forall r>0,
\end{equation}
which is nonnegative, nondecreasing and bounded by $\norm{\nabla u}{L^2(\Omega)}^2$. Take $\lambda\in (0,d)$. The first task is to prove that $r\mapsto r^{-\lambda}\ph(r)$ is bounded. Let $r_0:=\dist(\Omega',\partial\Omega)>0$ we first remark that for $r\geq r_0$, $r^{-\lambda}\ph(r)\leq r_0^{-\lambda}\norm{\nabla u}{L^2(\Omega)}^2$. For $r<r_0$ we have that $B_\Omega(x,r)=B(x,r)$ that we denote $B_r$ for simplicity. Since $a\geq \underline a$ almost everywhere, its mean value satisfies $a_{B_r}\geq \underline a$. As $u\in H^1(\Omega)$ is a weak solution of  \eqref{eq:main}, it satisfies
\begin{equation}\label{eq:decomp}
-a_{B_r}\Delta u = -\div((a_{B_r}-a)\nabla u)+\div F\quad  \tin B_r. 
\end{equation}
We decompose $u=u_1+u_2$ where $u_1\in H^1_0(B_r)$ is the unique solution  of homogeneous Dirichlet problem
\begin{equation}\label{eq:VF}
\left\{\begin{aligned}
-a_{B_r}\Delta u_1 &= -\div((a_{B_r}-a)\nabla u)+\div F\quad  \tin B_r,\\
u_1 &= 0\quad \ton \partial B_r, 
\end{aligned}\right.
\end{equation}
and $u_2$ is harmonic in $B_r$. For any $0<\rho<r$, we have
\begin{equation}\nonumber
\ph(\rho) \leq 2\int_{B_\rho}|\nabla u_1|^2 + 2\int_{B_\rho}|\nabla u_2|^2.
\end{equation}
As $u_2$ is harmonic in $B_r$, we use the decay speed result \ref{theo:decay} that says 
\begin{equation}\nonumber
\int_{B_\rho}|\nabla u_2|^2\leq C\frac{\rho^d}{r^d}\int_{B_r}|\nabla u_2|^2
\end{equation} 
where $C$ depends only on $d$. Moreover, 
\begin{equation}\nonumber
\begin{aligned}
\int_{B_r}|\nabla u_2|^2 &\leq 2\int_{B_r}|\nabla u|^2 + 2\int_{B_r}|\nabla u_1|^2\\
 &\leq 2\ph(r) + 2\int_{B_r}|\nabla u_1|^2.\\
\end{aligned}
\end{equation}
Combining the three precedent inequalities, it comes that 
\begin{equation}\label{eq:inter1}
\begin{aligned}
\ph(\rho) &\leq 2\int_{B_\rho}|\nabla u_1|^2 + 2C\frac{\rho^d}{r^d}\left(2\ph(r) + 2\int_{B_r}|\nabla u_1|^2\right)\\
& \leq 4C\frac{\rho^d}{r^d}\ph(r) + \left(2+4C\frac{\rho^d}{r^d}\right)\int_{B_r}|\nabla u_1|^2\\
&\leq A\left(\frac{\rho^d}{r^d}\ph(r) + \int_{B_r}|\nabla u_1|^2\right),
\end{aligned}
\end{equation}
with $A:=2+4C$ is a constant that depends only on $d$. 

We now compute a bound on $u_1$ based on the energy estimate of variational formulation of \eqref{eq:VF}. Using $u_1$ as a test function we get
\begin{equation}\nonumber
\begin{aligned}
a_{B_r}\int	_{B_r}|\nabla u_1|^2 &= \int_{B_r}(a_{B_r}-a)\nabla u\cdot \nabla u_1 -  \int_{B_r}F\cdot \nabla u_1 \\
&\leq \norm{a-a_{B_r}}{L^\infty(B_r)}\norm{\nabla u}{L^2(B_r)}\norm{\nabla u_1}{L^2(B_r)} + \norm{F}{L^2(B_r)}\norm{\nabla u_1}{L^2(B_r)}
\end{aligned}
\end{equation}
and so 
\begin{equation}\nonumber
\ul a \norm{\nabla u_1}{L^2(B_r)}\leq a_{B_r}\norm{\nabla u_1}{L^2(B_r)}\leq \norm{a-a_{B_r}}{L^\infty(B_r)}\norm{\nabla u}{L^2(B_r)} + \norm{F}{L^2(B_r)}. 
\end{equation}
Squaring this inequality, we obtain 
\begin{equation}\nonumber
\int_{B_r}|\nabla u_1|^2\leq \frac{2}{\ul a^2}\left(\norm{a-a_{B_r}}{L^\infty(B_r)}^2\ph(r)+\norm{F}{L^2(B_r)}^2\right). 
\end{equation}
Now us the hypothesis on $a$: 
\begin{equation}\nonumber
\norm{a-a_{B_r}}{L^\infty(B_r)}\leq r^\alpha \omega_\alpha[a](x)\leq 2r^\alpha g(x),
\end{equation}
and from the definition of the $E_\Omega^{2,\lambda}[F]$, we write 
\begin{equation}\nonumber
\norm{F}{L^2(B_r)}\leq r^{\lambda/2}E_\Omega^{2,\lambda}[F](x). 
\end{equation}

Using these two controls we get 
\begin{equation}\label{eq:inter2}
\int_{B_r}|\nabla u_1|^2\leq \frac{2}{\ul a^2}\left(4r^{2\alpha}g(x)^2\ph(r)+r^\lambda E_\Omega^{2,\lambda}[F](x)^2\right). 
\end{equation}
Combining now \eqref{eq:inter1} and  \eqref{eq:inter2} we get 
\begin{equation}\nonumber
\ph(\rho)\leq A\left(\frac{\rho^d}{r^d}\ph(r) +Br^{2\alpha}\ph(r)\right) + Cr^\lambda,\quad 0<\rho\leq r\leq r_0. 
\end{equation}
with $A$ depending only on $d$, $B:= 8 g(x)^2/\ul a^2$ and $C:=2A (E_\Omega^{2,\lambda}[F](x))^2/\ul a^2$. From such inequality, the the iteration lemma (see Appendix \ref{app:iteration})shows that the function $r^{-\lambda}\ph(r)$ is bounded in $(0,r_0]$. More precisely, in this case it gives 
\begin{equation}\nonumber
r^{-\lambda}\ph(r)\leq D_1\left(r_0^{-\lambda}\ph(r_0) +D_2B^{\lambda/(2\alpha)}\ph(r_0)+ C\right),\quad \forall r\in(0,r_0].
 \end{equation}
where $D_1:=D_1(d,\lambda)$ and $D_2:=D_2(d,\alpha,\lambda)$. Taking the square root of this last control and recalling that as $r<r_0$, $B_{\Omega'}(x,r)\subset B_{\Omega}(x,r) \subset B_r$, 
\begin{equation}\nonumber
\begin{aligned}
r^{-\lambda/2}\norm{\nabla u}{L^2(B_{\Omega'}(x,r))} &\leq D_1^{1/2}\left(\left(r_0^{-\lambda/2}+D_3g(x)^{\lambda/(2\alpha)}\right)\norm{\nabla u}{L^2(\Omega)}+ D_4E_\Omega^{2,\lambda}[F](x)\right)\\
\end{aligned}
\end{equation}
where $D_3:=D_3(d,\alpha,\lambda,\ul a)$ and $D_4:=D_4(d,\ul a)$ are positive constants. Remarking that for $r\geq r_0$ we have
\begin{equation}\nonumber
r^{-\lambda/2}\norm{\nabla u}{L^2(B_{\Omega'}(x,r))} \leq r_0^{-\lambda/2}\norm{\nabla u}{L^2(\Omega)},\quad r_0< r< +\infty,
\end{equation}
we can take the supremum over $r\in (0,+\infty)$ to get 
\begin{equation}\nonumber
E_{\Omega'}^{2,\lambda}[\nabla u](x) \leq D_1^{1/2}\left(\left(r_0^{-\lambda/2}+D_3g(x)^{\lambda/(2\alpha)}\right)\norm{\nabla u}{L^2(\Omega)}+ D_4E_\Omega^{2,\lambda}[F](x)\right). 
\end{equation}
As this last control is valid for almost every $x\in\Omega'$, one can take $L^{q'}$-norm with $q'\leq 2\alpha q/\lambda$ in $x$ over $\Omega'$ to get when $q'<+\infty$
\begin{equation}\nonumber
\norm{\nabla u}{L^{2,\lambda,q'}(\Omega')} \leq D_1^{1/2}\left(\left(r_0^{-\lambda/2}|\Omega|^{1/q'}+D_5\norm{g}{L^q(\Omega)}^{\lambda/(2\alpha)}\right)\norm{\nabla u}{L^2(\Omega)}+ D_4\norm{F}{L^{2,\lambda,q'}(\Omega)}\right). 
\end{equation}
where $D_5:=D_5(d,\alpha,\lambda,\ul a,q',q,\Omega)$.  This is stil valid even if  $q=q'=+\infty$. We deduce that $\nabla u\in L^{2,\lambda,q'}(\Omega')$ as well as the demanded final control. 
\end{proof}

\begin{remark}
In the classical Schauder--Campanato theory, the coefficient $a$ is assumed
to belong to $C^{0,\alpha}\big(\ol\Omega\big )$, which corresponds to the case
$g\in L^\infty(\Omega)$ in hypothesis $(H_2)$.
Then one can take $q'=+\infty$, recovering the classical
transmission of Morrey regularity from $F$ to $\nabla u$. The present result shows that the same mechanism remains valid when the
Hölder seminorm of the coefficient is no longer uniformly bounded but only
belongs to $L^q(\Omega)$ through the oscillation control $(H_2)$.
\end{remark}

\begin{remark} It is interesting to remark that
the restriction $\lambda\leq 2\alpha q$ naturally reflects the
integrability of the oscillation $\omega_\alpha[a]$ of the coefficient.
Indeed, the proof requires the quantity
$g^{\lambda/(2\alpha)}$ to belong to $L^{q'}(\Omega)$, which is equivalent
to the condition $q'\leq 2\alpha q/\lambda$ an that is possible only if $\lambda\leq 2\alpha q$ as $q'\geq 1$. 
\end{remark}

\subsection{Estimate in generalized Campanato spaces}

We now propose an estimate in generalized Campanato spaces. Once again, the proof follows the classical Morrey--Campanato approach for elliptic equations with replacement of the global boundedness of the oscillations by suitable integrability argument.

\begin{theorem}\label{theo:2} Let $u\in H^1(\Omega)$ be a weak solution of \eqref{eq:main} with the coefficient $a$ satisfying the hypotheses $(H_1)$ and $(H_2)$ with $q\in [1,+\infty]$ and $\alpha\in (0,1]$. For $\Omega'\Subset \Omega$ and any $\lambda\in(0,d]$ with $q'\in [1,+\infty]$, for any $1\leq q''\leq (1/q+1/q')^{-1}$, then 
\begin{equation}\nonumber
\nabla u\in L^{2,\lambda,q'}(\Omega)\quad  \text{and}\quad F\in L_C^{2,\lambda+2\alpha,q''}(\Omega)\quad\imp\quad \nabla u\in L_C^{2,\lambda+2\alpha,q''}(\Omega'),
\end{equation}
 and there exists a positive constant $C:=C(d,\alpha,\lambda,q,q',q'',\underline{a},\Omega,\Omega')$ such that
  \begin{equation}\nonumber
 \norm{\nabla u}{L_C^{2,\lambda+2\alpha,q''}(\Omega')}\leq C\left(\norm{\nabla u}{L^2(\Omega)} + \norm{g}{L^q(\Omega)}\norm{\nabla u}{L^{2,\lambda,q'}(\Omega)}+ \norm{F}{L_C^{2,\lambda+2\alpha,q''}(\Omega)}\right).
 \end{equation}
\end{theorem}

\begin{proof} From the hypotheses on $a$, $F$ and $\nabla u$ we know that almost everywhere in $\Omega$ we have $\omega_\alpha[a]\leq 2g$, $g<+\infty$, $O_\Omega^{2,\lambda+2\alpha}[F]<+\infty$ and $E_\Omega^{2,\lambda}[\nabla u]<+\infty$. We then fix $x\in\Omega'$ such that these four inequalities are fulfilled at point $x$.

We start by defining the function 
\begin{equation}\nonumber
\psi(r):=\int_{B_\Omega(x,r)}|\nabla u-(\nabla u)_{B_\Omega(x,r)}|^2,\quad \forall r>0,
 \end{equation}
which is nonnegative. It is also nondecreasing as if $0<\rho\leq r<+\infty$, 
\begin{equation}\nonumber
\begin{aligned}
\psi(\rho) &=\int_{B_\Omega(x,\rho)}|\nabla u-(\nabla u)_{B_\Omega(x,\rho)}|^2=\min_{V\in\R^d}\int_{B_\Omega(x,\rho)}|\nabla u-V|^2\\
&\leq \int_{B_\Omega(x,\rho)}|\nabla u-(\nabla u)_{B_\Omega(x,r)}|^2\leq \int_{B_\Omega(x,r)}|\nabla u-(\nabla u)_{B_\Omega(x,r)}|^2\\
& \leq \psi(r).
\end{aligned}
\end{equation}
The first task is to prove that $r\mapsto r^{-\lambda-2\alpha}\psi(r)$ is bounded. Let $r_0:=\dist(\Omega',\partial\Omega)>0$ we first remark that for $r>r_0$, we can write $r^{-\lambda-2\alpha}\psi(r)\leq r_0^{-\lambda-2\alpha}\norm{\nabla u - (\nabla u)_\Omega}{L^2(\Omega)}^2$. Consider now $0<r\leq r_0$, then $B_\Omega(x,r)=B(x,r)$ denoted $B_r$ for simplicity and then 
\begin{equation}\nonumber
\psi(r)=\int_{B_r}|\nabla u - (\nabla u)_{B_r}|^2,\quad \forall r\in (0,r_0],
\end{equation}

Define now the function $v(y):=u(y)-(\nabla u)_{B_r}\cdot y$, it satisfies $\nabla v = \nabla u -(\nabla u)_{B_r}$ and $\Delta v = \Delta u$. From \eqref{eq:decomp} it satisfies then
\begin{equation}\nonumber
-a_{B_r}\Delta v = -\div((a_{B_r}-a)\nabla u)+\div (F-F_{B_r})\quad  \tin B_r. 
\end{equation}
We decompose $v$ as $v=v_1+v_2$ where $v_1\in H^1_0(B_r)$ is the unique solution of the elliptic problem 
\begin{equation}\label{eq:VF2}
\left\{\begin{aligned}
-a_{B_r}\Delta v_1 &= -\div((a_{B_r}-a)\nabla u)+\div (F-F_{B_r}),\quad  \tin B_r,\\
v_1 &= 0\quad \ton \partial B_r, 
\end{aligned}\right.
\end{equation}
and $v_2$ is harmonic in $B_r$. Call now for $0<\rho\leq r$,
\begin{equation}\nonumber
\psi_1(\rho):= \int_{B_\rho}|\nabla v_1-(\nabla v_1)_{B_\rho}|^2\quad \text{ and }\quad \ \psi_2(\rho):= \int_{B_\rho}|\nabla v_2-(\nabla v_2)_{B_\rho}|^2.
\end{equation}
It is clear that $\psi(\rho)\leq 2\psi_1(\rho)+2\psi_2(\rho)$ for any $\rho\in(0,r]$. Let us bound these two terms independently. 

First
\begin{equation}\nonumber
\psi_1(r)=\int_{B_r}|\nabla v_1-(\nabla v_1)_{B_r}|^2\leq \int_{B_r}|\nabla v_1|^2.
\end{equation}
From the variational formulation of \eqref{eq:VF2} we get
\begin{equation}\nonumber
\begin{aligned}
a_{B_r}\int_{B_r}|\nabla v_1|^2 &=\int_{B_r}(a_{B_r}-a)(\nabla v+(\nabla u)_{B_r})\cdot\nabla v_1+\int_{B_r}(F-F_{B_r})\cdot \nabla v_1,\\
&\leq \norm{a-a_{B_r}}{L^\infty(B_r)}\norm{\nabla u}{L^2(B_r)}\norm{\nabla v_1}{L^2(B_r)} +\norm{F-F_{B_r}}{L^2(B_r)}\norm{\nabla v_1}{L^2(B_r)},
\end{aligned}
\end{equation}
and then 
\begin{equation}\nonumber
\begin{aligned}
\ul a \norm{\nabla v_1}{L^2(B_r)} &\leq \norm{a-a_{B_r}}{L^\infty(B_r)}\norm{\nabla u}{L^2(B_r)} +\norm{F-F_{B_r}}{L^2(B_r)}.\\
\ul a^2 \norm{\nabla v_1}{L^2(B_r)}^2 &\leq 2\norm{a-a_{B_r}}{L^\infty(B_r)}^2\norm{\nabla u}{L^2(B_r)}^2 +2\norm{F-F_{B_r}}{L^2(B_r)}^2. 
\end{aligned}
\end{equation}
By the definition of $\omega_\alpha[a]$, $E_\Omega^{2,\lambda}[\nabla u]$ and $O_\Omega^{2,\lambda+2\alpha}[F]$ we write
\begin{equation}\nonumber
\left\{\begin{aligned}
\norm{a-a_{B_r}}{L^\infty(B_r)} &\leq r^{\alpha}\omega_\alpha[a](x)\leq 2 r^\alpha g(x)\\
\norm{\nabla u}{L^2(B_r)} &\leq r^{\lambda/2}E_\Omega^{2,\lambda}[\nabla u](x)\\
\norm{F-F_{B_r}}{L^2(B_r)} &\leq r^{\lambda/2+\alpha}O_\Omega^{2,\lambda+2\alpha}[F](x)
\end{aligned}\right.
\end{equation}
and deduce
\begin{equation}\nonumber
\psi_1(r)\leq \frac{2}{\ul a^2}r^{\lambda+2\alpha}\left(4g(x)^2(E_\Omega^{2,\lambda}[\nabla u](x))^2+ (O_\Omega^{2,\lambda+2\alpha}[F](x))^2\right).
\end{equation}

As $v_2$ is harmonic, using the decay speed result (Theorem \ref{theo:decay}), there exists a constant $\kappa:=\kappa(d)>0$ such that 
\begin{equation}\nonumber
\psi_2(\rho) \leq \kappa\frac{\rho^{d+2}}{r^{d+2}}\psi_2(r),\quad \forall \rho\in(0,r]. 
\end{equation}
Now considering that $\psi_2(r)\leq 2\psi(r)+2\psi_1(r)$ it comes as $\rho/r\leq 1$, 
\begin{equation}\nonumber
\begin{aligned}
\psi_2(\rho) &\leq 2\kappa\frac{\rho^{d+2}}{r^{d+2}}\psi(r) + 2\kappa\psi_1(r),\quad \forall\rho\in(0,r],\\
\end{aligned}
\end{equation}
and so for all $\rho\in(0,r]$,
\begin{equation}\label{eq:inter3}
\begin{aligned}
\psi(\rho) &\leq 2\psi_1(\rho)+2\psi_2(\rho)\\
&\leq 4\kappa\frac{\rho^{d+2}}{r^{d+2}}\psi(r) + (4\kappa+2)\psi_1(r)\\
&\leq A\frac{\rho^{d+2}}{r^{d+2}}\psi(r) + Cr^{\lambda+2\alpha},
\end{aligned}
\end{equation}
where $A:=4\kappa$ depends only on $d$ and 
\begin{equation}\nonumber
C:=\frac{4(2\kappa+1)}{\ul a^2}\left(4g(x)^2(E_\Omega^{2,\lambda}[\nabla u](x))^2+ (O_\Omega^{2,\lambda+2\alpha}[F](x))^2\right). 
\end{equation}
Applying now the iteration lemma \ref{lem:it} to the inequality obtained in \eqref{eq:inter3} we get that as $\lambda+2\alpha<d+2$, 
 \begin{equation}\nonumber
r^{-\lambda-2\alpha}\psi(r)\leq D_1\left(r_0^{-\lambda-2\alpha}\psi(r_0) + C\right)\quad \forall r\in(0,r_0],
 \end{equation}
 where $D_1:=D_1(d,\lambda,\alpha)$. Taking the square root of this expression, we get for all $r\in (0,r_0]$,
 \begin{equation}\nonumber
 \begin{aligned}
  &r^{-\lambda/2-\alpha}\norm{\nabla u-(\nabla u)_{B_{\Omega'}(x,r)} }{L^2(B_{\Omega'}(x,r))} \leq r^{-\lambda/2-\alpha}\psi(r)^{1/2} \\ &\leq D_1^{1/2}\left(r_0^{-\lambda/2-\alpha}\norm{\nabla u-(\nabla u)_\Omega}{L^2(\Omega)} + 2D_2 g(x)E_\Omega^{2,\lambda}[\nabla u](x)+ D_2O_\Omega^{2,\lambda+2\alpha}[F](x)\right)
 \end{aligned}
 \end{equation}
 where $D_2:=2\sqrt{2\kappa+1}/\ul a$. This inequality is still valid for $r\geq r_0$. Taking the supremum over $r>0$, we arrive at
 \begin{equation}\nonumber
\begin{aligned}
O_{\Omega'}^{2,\lambda+2\alpha}[\nabla u](x)
&\leq D_1^{1/2}\Big(
2r_0^{-\lambda/2-\alpha}
\norm{\nabla u}{L^2(\Omega)} 
+ 2D_2 g(x)E_\Omega^{2,\lambda}[\nabla u](x)
+ D_2O_\Omega^{2,\lambda+2\alpha}[F](x)
\Big).
\end{aligned}
\end{equation}
which is valid for almost every $x\in\Omega'$. Now taking $q''\in [1,+\infty]$ such that $1/q'' \geq 1/q+1/q'$ we take the $L^{q''}$--norm over $\Omega'$ and use Hölder:
\begin{equation}\nonumber
\begin{aligned}
\norm{O_{\Omega'}^{2,\lambda+2\alpha}[\nabla u]}{L^{q''}(\Omega')}
\leq\;
&D_1^{1/2}\Big(
D_3\norm{\nabla u}{L^2(\Omega)} \\
&\qquad + 2D_2 |\Omega|^{\frac 1s} \norm{g}{L^q(\Omega)}
\norm{\nabla u}{L^{2,\lambda,q'}(\Omega)}
+ D_2\norm{F}{L_C^{2,\lambda+2\alpha,q''}(\Omega)}
\Big).
\end{aligned}
\end{equation}
 where $1/s=1/q''-1/q-1/q'$ and $D_3:=D_3(r_0,\alpha,\lambda,\Omega,q'')$. Remembering that
\begin{equation}\nonumber
\norm{\nabla u}{L_C^{2,\lambda+2\alpha,q''}(\Omega')}
:=
\norm{\nabla u}{L^{2}(\Omega')}
+
\norms{O_{\Omega'}^{2,\lambda+2\alpha}[\nabla u]}{L^{q''}(\Omega')}
\end{equation}
and that $\norm{\nabla u}{L^2(\Omega')}
\leq
\norm{\nabla u}{L^2(\Omega)}$,
we obtain the desired estimate and conclude.
\end{proof}

\begin{remark}
Combining Theorems \ref{theo:1} and \ref{theo:2}, one obtains an interesting 
regularity transfer mechanism. Theorem \ref{theo:1} first provides a
generalized Morrey estimate for $\nabla u$ from the regularity of $a$ and $F$.
This estimate can then be inserted into Theorem \ref{theo:2} to obtain
a stronger generalized Campanato estimate for $\nabla u$. This strategy is the
basis of the estimates derived in the next section.
\end{remark}

\begin{remark}
In the critical case $\lambda=d$, the generalized Morrey space
$L^{2,d,q'}(\Omega)$ coincides with $L^{q'}(\Omega)$.
Consequently, Theorem \ref{theo:2} may be interpreted in this case as a generalized Campanato
regularity estimate assuming only an $L^{q'}$ control of the gradient. Nevertheless, note that the case $\lambda=d$ cannot be directly obtained from \ref{theo:1}. 
\end{remark}

\begin{remark}
A notable feature of Theorem \ref{theo:2} is that the oscillation of the
coefficient appears through the product
\[
g\, E_\Omega^{2,\lambda}[\nabla u].
\]
The admissible exponent $q''$ is therefore determined by the Hölder
compatibility condition
\[
\frac1{q''}\geq \frac1q+\frac1{q'}.
\]
Hence, applying Theorem \ref{theo:2} allows one to increase the
Campanato parameter from $\lambda$ to $\lambda+2\alpha$.
The price to pay for this improvement is a possible decrease of the
integrability exponent, from $q'$ to $q''\leq q'$.
\end{remark}

\section{Deduced gradient estimates}

In this section, we make use of the regularity transfer mechanisms provided by
Theorems \ref{theo:1} and \ref{theo:2} to derive a series of regularity estimates in
Hölder, Lebesgue and fractional Sobolev spaces. To do so, we repeatedly use
the embedding rules established in Propositions \ref{prop:injmorrey} and
\ref{prop:injcamp} without recalling the precise references in each proof.

With this first corollary, we recover the classical Schauder estimate, which
can be found for instance in
\cite{gilbarg1998elliptic,giaquinta2013introduction,troianiello2013elliptic}.
It relies on successive applications of the two main theorems in the particular
case $q=q'=q''=+\infty$.

\begin{corollary}[Schauder estimate]\label{cor:Holder} If the coefficient $a$ satisfies $(H_1)$ in $\Omega$, for $\alpha\in (0,1)$ and $\Omega'\Subset \Omega$ being a $\theta$-smooth subdomain, we have
\begin{equation}\nonumber
a\in \cC^{0,\alpha}\big(\overline\Omega\big)\quad \text{and}\quad  F\in \cC^{0,\alpha}\big(\overline\Omega\big)\quad\imp\quad \nabla u\in \cC^{0,\alpha}\left(\ol{\Omega'}\right). 
\end{equation}
\end{corollary}

\begin{proof} We successively apply Theorem \ref{theo:1} and then Theorem \ref{theo:2} twice. Let $\Omega'=\Omega_3\Subset\Omega_2\Subset\Omega_1\Subset\Omega$ be some intermediate $\theta$-smooth subdomains. 

As $a\in \cC^{0,\alpha}\left(\ol\Omega\right)$, it satisfies the hypothesis $(H_2)$ with $\alpha$ and $q=+\infty$. As $F\in \cC^{0,\alpha}\big(\overline\Omega\big)\sim L_C^{2,d+2\alpha,\infty}(\Omega)\inj L_C^{2,d-\alpha,\infty}(\Omega) \sim L^{2,d-\alpha,\infty}(\Omega)$, we first apply Theorem \ref{theo:1} with $q=+\infty$ and $\lambda = d-\alpha$. We get that $\nabla u\in L^{2,d-\alpha,\infty}(\Omega_1)$.

As $F\in L_C^{2,d+2\alpha,\infty}(\Omega)\inj L_C^{2,d+\alpha,\infty}(\Omega)$, and $\nabla u\in L^{2,d-\alpha,\infty}(\Omega_1)$, we apply Theorem \ref{theo:2} in $\Omega_1$ with $\lambda=d-\alpha$ to get that $\nabla u\in L_C^{2,d+\alpha,\infty}(\Omega_2)\sim \cC^{0,\alpha/2}\big(\overline{\Omega_2}\big)\inj L^\infty(\Omega_2)\sim L^{2,d,\infty}(\Omega_2)$. 

As $F\in L_C^{2,d+2\alpha,\infty}(\Omega)$ and $\nabla u\in L_C^{2,d,\infty}(\Omega_2)$, we re-apply Theorem \ref{theo:2} in $\Omega_2$ with $\lambda=d$ to deduce that $\nabla u\in L_C^{2,d+2\alpha,\infty}(\Omega_3)\sim \cC^{0,\alpha}\big(\overline{\Omega_3}\big)$.
\end{proof}

\begin{corollary}[$L^r$ regularity for $\alpha q>d$]\label{cor:Lr}
Assume that the coefficient $a$ satisfies $(H_1)$ and $(H_2)$ with
$\alpha q>d$ in $\Omega$.
Let $\Omega'\Subset\Omega$ be a $\theta$-smooth subdomain.
Then
\[
\forall r\in(2,+\infty),\qquad
F\in L^r(\Omega)
\quad\imp\quad
\nabla u\in L^r(\Omega').
\]
\end{corollary}

\begin{proof} Let $\Omega'=\Omega_2\Subset\Omega_1\Subset\Omega$ be some $\theta$-smooth subdomains. As $\alpha q>d$ we stand in the situation where $a$ is Hölder continuous. More precisely, $a\in \cC^{0,\eta}\left(\ol\Omega\right)$ with $\eta<\alpha-d/q$ (see Proposition \ref{prop:H2impSobolev}). Then $a$ satisfies $(H_2)$ with new exponents $\eta<\alpha-d/q$ and $q'=+\infty$. As $F\in L^r(\Omega)\sim L^{2,d,r}(\Omega)$, we apply Theorem \ref{theo:1} to get that $\nabla u\in L^{2,\lambda,r}(\Omega_1')$ for any $\lambda\in(0,d)$. 

As $F\in L^r(\Omega)\sim L_C^{2,d,r}(\Omega)$, and choosing $\lambda$ and $\alpha'$ such that $\lambda+2\alpha'=d$, (possible if we choose $\alpha'$ small enough) we apply now Theorem \ref{theo:2} to get that $\nabla u\in L_C^{2,\lambda+2\alpha',r}(\Omega_2') = L_C^{2,d,r}(\Omega_2')\sim L^r(\Omega_2')$. 
\end{proof}

\begin{remark}
In the regime $\alpha q>d$, Proposition \ref{prop:H2impSobolev} implies
that the coefficient $a$ is H\"older continuous. Therefore, the previous
$L^r$ estimate is consistent with the classical $W^{1,r}$ regularity
theory for elliptic equations in divergence form with sufficiently regular
coefficients. The interest here is that the estimate is recovered directly
from the generalized Morrey--Campanato transfer mechanism: Theorem
\ref{theo:1} first gives a generalized Morrey estimate for $\nabla u$, and
Theorem \ref{theo:2} then upgrades it to the critical Campanato space
$L_C^{2,d,r}$, which is equivalent to $L^r$.
\end{remark}

The next corollary shows that in the regime $\alpha q>d$, Theorems~\ref{theo:1} and \ref{theo:2} yields a
local Sobolev regularity result: if $F\in W^{\alpha,q}(\Omega)$, then
$\nabla u\in W^{\alpha',q}(\Omega')$ for every $\alpha'<\alpha$.

\begin{corollary}[Sobolev regularity for $\alpha q>d$] If the coefficient $a$ satisfies $(H_1)$ and $(H_2)$ with $\alpha q>d$ and $q\geq 2$ in $\Omega$, for $\Omega'\Subset \Omega$ being a $\theta$-smooth subdomain, we have
\begin{equation}\nonumber
 F\in W^{\alpha,q}(\Omega)\quad\imp\quad \nabla u\in W^{\alpha',q}(\Omega'),\quad \forall \alpha'<\alpha.  
\end{equation}
\end{corollary}

\begin{proof} Let $\Omega'=\Omega_2\Subset\Omega_1\Subset\Omega$ be some $\theta$-smooth subdomains.  As $\alpha q>d$ we stand in the situation where $a$ is Hölder continuous. More precisely, $a\in \cC^{0,\eta}\left(\ol\Omega\right)$ with $\eta<\alpha-d/q$ (see Proposition \ref{prop:H2impSobolev}). Moreover, as $F\in W^{\alpha,q}(\Omega)$, we also have $F\in \cC^{0,\eta}\left(\ol\Omega\right)$. Using Corollary \ref{cor:Holder}, we get that $\nabla u\in \cC^{0,\eta}\big(\ol{\Omega_1'}\big)$ for $\Omega_1'\Subset\Omega$ being $\theta$-smooth. Hence $\nabla u\in L^\infty(\Omega_1') \sim L^{2,d,\infty}(\Omega_1)$. 

As $F\in W^{\alpha,q}(\Omega)\inj L_C^{2,d+2\alpha,q}(\Omega)$ as $q\geq 2$ and $\nabla u\in  L^{2,d,\infty}(\Omega_1)$ we apply Theorem \ref{theo:2} in $\Omega_1'$ with $\lambda=d$ to get that $\nabla u\in L_C^{2,d+2\alpha,q}(\Omega_2')\inj W^{\alpha',q}(\Omega_2')$ for any $\alpha'<\alpha$. 
\end{proof}

The previous corollary can be extended to arbitrary Sobolev source term $F$. More precisely, if $F\in W^{\beta,r}(\Omega)$, then the regularity of $\nabla u$ is determined by the interplay between the regularity of the coefficient and that of the source term.

\begin{corollary}[General Sobolev regularity for $\alpha q>d$] If the coefficient $a$ satisfies $(H_1)$ and $(H_2)$ with $\alpha q> d$ and $q\geq 2$ in $\Omega$, and if $F\in W^{\beta,r}(\Omega)$ with $r\geq 2$ and $\beta>0$, call
\begin{equation}\nonumber
\beta_0:=\min(\alpha,\beta),\quad r':=
\left\{\begin{aligned}
&\frac{d r}{d-r(\beta-\beta_0)}\quad &&\text{if } r(\beta-\beta_0)<d,\\
&+\infty\quad &&\text{else} \hfill
\end{aligned}\right.
\quad\text{and}\quad r_0:=\min\left(r',\frac{\alpha q}{\beta_0}\right),
\end{equation}
then for $\Omega'\Subset \Omega$ being a $\theta$-smooth subdomain,  the following holds:  
\begin{itemize}

\item If $\beta r>d$ we have $\nabla u\in W^{\beta',r_0}(\Omega')$ for any $\beta' <\beta_0$.

\item If $\beta r=d$ we have $\nabla u\in W^{\beta',\rho}(\Omega')$ for any $\beta' <\beta_0$ and $\rho<r_0$.

\item If $\beta r<d$ we have $\nabla u\in W^{\beta',r'}(\Omega')$ for any $\beta' <\beta_0$.

\end{itemize}

\end{corollary}

\begin{proof} Let $\Omega'=\Omega_2\Subset\Omega_1\Subset\Omega$ be some $\theta$-smooth subdomains. As $\alpha q>d$ we stand in the situation where $a$ is Hölder. More precisely, $a\in \cC^{0,\eta}\left(\ol\Omega\right)$ with $\eta<\alpha-d/q$ (see Proposition \ref{prop:H2impSobolev}). Define now $\beta_0:=\min(\alpha,\beta)$, from Proposition \ref{prop:H2impH2} the coefficient  $a$ satisfies $(H_2)$ with exponents $\beta_0$ and $q'=\alpha q/\beta_0\geq q\geq 2$. Moreover, by Sobolev embedding, $F\in W^{\beta_0,r'}(\Omega_1)$. 

Case $\beta r>d$: In this case, $F$ is also Hölder in $\Omega$ and then, from Corollary \ref{cor:Holder}, we get that $\nabla u$ is Holder in $\Omega_1$ and then $\nabla u\in L^\infty(\Omega_1)\sim L^{2,d,\infty}(\Omega_1)$.  As the paremeter $a$ satisfies $(H_2)$ with exponents $\beta_0$ and $q'\geq r_0$, $F\in L_C^{2,d+2\beta_0,r'}(\Omega_1)\inj L_C^{2,d+2\beta_0,r_0}(\Omega_1)$ and $\nabla u\in L^{2,d,\infty}(\Omega_1)$, we apply Theorem \ref{theo:2} to get that $\nabla u\in L_C^{2,d+2\beta_0,r_0}(\Omega_2)\inj W^{\beta',r_0}(\Omega_2)$ for all $\beta'< \beta_0$. 

Case $\beta r=d$: In this case $F\in L^{s}(\Omega)$ for any $s\in [1,+\infty)$ by Sobolev embedding, and then from Corollary \ref{cor:Lr}, $\nabla u\in L^s(\Omega_1)\sim L^{2,d,s}(\Omega_1)$ for any $2\leq s<+\infty$. As the paremeter $a$ satisfies $(H_2)$ with exponents $\beta_0$ and $r_0$, $F\in L_C^{2,d+2\beta_0,r_0}(\Omega_1)$ and $\nabla u\in L^{2,d,s}(\Omega_1)$, we apply Theorem \ref{theo:2} to get that $\nabla u\in L_C^{2,d+2\beta_0,r''}(\Omega_2)\inj W^{\beta',r''}(\Omega_2)$ with $r''=r_0s/(s+r_0)$ for all $\beta'< \beta_0$ and all $s\geq 2$.   

Case $\beta r<d$: In this case $F\in L^{s}(\Omega)$ for $s=(1/r-\beta/d)^{-1}$. With the same argumentation than in the previous case, we get $\nabla u\in L^{2,d,s}(\Omega_1)$. Now as the coefficient  $a$ satisfies $(H_2)$ with exponents $\beta_0$ and $q'$, we call $r''=(1/q'+1/s)^{-1}$. We remark that $r'< r''$, indeed
\begin{equation}\nonumber
\frac{1}{r'} - \frac{1}{r''} = \frac{1}{r} - \frac{\beta-\beta_0}{d} - \frac{1}{q'}-\frac{1}r+\frac{\beta}{d} =    \frac{\beta_0}{d} - \frac{1}{q'} =   \beta_0\left(\frac{1}{d} - \frac{1}{\alpha q}\right)>0. 
\end{equation}
As $F\in W^{\beta_0,r'}(\Omega_1)\inj L_C^{2,d+2\beta_0,r'}(\Omega_1)$, with $r'<r''$ we get by Theorem \ref{theo:2} that $\nabla u\in L_C^{2,d+2\beta_0,r'}(\Omega_2)\inj W^{\beta',r'}(\Omega_2)$. 
\end{proof}

\begin{remark}
The exponent $r'$ corresponds to the Sobolev embedding
$W^{\beta,r}(\Omega)\hookrightarrow W^{\beta_0,r'}(\Omega)$.
Consequently, the regularity index $\beta_0=\min(\alpha,\beta)$ reflects
the fact that the regularity transferred to $\nabla u$ cannot exceed either
the regularity of the source term or the oscillation exponent of the
coefficient.
\end{remark}

The next result is perhaps the most interesting consequence of the present work, since it applies in
the regime $\alpha q\in(d/2,d]$, where the coefficient $a$ may not to be
continuous. In this range, one can still obtain a fractional Sobolev
regularity estimate for $\nabla u$. As expected, this regularity remains
low, and the resulting Sobolev regularity do not yield any $L^r$ estimate
with $r>2$ for $\nabla u$.

\begin{corollary}[Sobolev estimate for $d/2<\alpha q \leq d$]\label{cor:low} If the coefficient $a$ satisfies $(H_1)$ and $(H_2)$ with $\alpha q \in(d/2,d]$ and $q\geq 2$ in $\Omega$, for $\Omega'\Subset \Omega$ being a $\theta$-smooth subdomain and $F\in W^{\alpha,q}(\Omega)$ we have
\begin{equation}\nonumber
\nabla u\in W^{\beta',\frac{2\alpha q}{d+2\beta}}(\Omega_2'),\quad \forall \beta\in(0,\alpha)\cap [\beta_{\min},\beta_{\max}],\quad\forall \beta'<\beta, 
\end{equation} 
where 
\begin{equation}\nonumber
\beta_{\min}:=\frac{3\alpha -d}{2}\qquad\text{and}\qquad \beta_{\max}:=\alpha q-\frac{d}{2}. 
\end{equation}
\end{corollary}

\begin{proof} Let $\Omega'=\Omega_2\Subset\Omega_1\Subset\Omega$ be some $\theta$-smooth subdomains. Let $\lambda\in (0,d)$. As $\alpha q\geq d/2$ we have $\lambda< 2\alpha q$. As $F\in W^{\alpha,q}(\Omega)\inj L^{q'}(\Omega)$ with $q'=q d/(d-\alpha q)>q\geq 2$, we have $F\in L^{2,d,q'}(\Omega)\inj  L^{2,\lambda,q'}(\Omega)$. If one choose $\lambda\geq \alpha$ then $q'\geq 2\alpha q/\lambda$. Indeed
\begin{equation}\nonumber
\begin{aligned}
\frac{\lambda}{2\alpha q} - \frac{1}{q'} \geq \frac{1}{2q} - \frac{1}{q} + \frac{\alpha}{d} =  \frac{\alpha}{d} - \frac{1}{2q} = \frac{2\alpha q-d}{2dq}\geq 0,
\end{aligned}
\end{equation}
and hence, $F\in L^{2,\lambda,2\alpha q/\lambda }(\Omega)$. Applying now Theorem \ref{theo:1}, we get that $\nabla u\in L^{2,\lambda,2\alpha q/\lambda }(\Omega_1')$. 

Call now $q''=(1/q+\lambda/(2\alpha q))^{-1} = 2\alpha q/(\lambda+2\alpha)$, as $F\in W^{\alpha,q}(\Omega)\inj L_C^{2,d+2\alpha,q}(\Omega)$, with $q>q''$, we have $F\in L_C^{2,\lambda+2\alpha,q''}(\Omega)$ (from gneralized Camapanto embedding rules) and therefore, using Theorem \ref{theo:2}, we get $\nabla u\in L_C^{2,\lambda+2\alpha,q''}(\Omega_2')$. This is valid for any $\lambda\in [\alpha,d)$. In order to have $q''\geq 1$ we need also to impose $\lambda\leq 2\alpha(q-1)$. 

Define now $\beta:=\alpha+(\lambda-d)/2$ in order to have $\lambda+2\alpha = d+2\beta$. Hence 
\begin{equation}\nonumber
\begin{aligned}
\nabla u\in L_C^{2,d+2\beta,\frac{2\alpha q}{d+2\beta}}(\Omega_2')
\end{aligned}
\end{equation}
The various constraints on $\lambda$ lead to constraints on $\beta$:
\begin{equation}\nonumber
\begin{aligned}
\lambda \geq \alpha\quad &\imp\quad \beta\geq (3\alpha - d)/2,\\
\lambda <d \quad &\imp\quad \beta <\alpha,\\
\lambda \leq 2\alpha(q-1) \quad &\imp\quad \beta \leq \alpha q- d/2,
\end{aligned}
\end{equation} 
and we add the constraint $\beta>0$ in order to get an injection in a Sobolev space. Hence for $\beta$ satisfiying all these constraints we get 
\begin{equation}\nonumber
\nabla u\in W^{\beta',\frac{2\alpha q}{d+2\beta}}(\Omega_2'),\quad \forall \beta'<\beta. 
\end{equation}
\end{proof}

\begin{remark} Unlike the previous corollaries, this result applies in a regime where the
coefficient $a$ is not expected to be Hölder continuous and may even be
discontinuous as $\alpha q\le d$. In particular, neither Schauder theory
nor the $L^r$ estimates of Corollary~\ref{cor:Lr} are available. Nevertheless,
one still obtains a positive fractional Sobolev regularity for $\nabla u$.
\end{remark}

\section*{Acknowledgement}
The author acknowledge support from the French National Research Agency (ANR) under
grants ANR-22-CE40-0005 (project REWARD).

\begin{appendix}

\section{Local mean function and $(p,\lambda)$--oscillation}\label{app:localmean}

From the Lebesgue differentiation Theorem, we already know that for any $u\in L^p(\Omega)$, the sequence of the local means $(u_r)_{r>0}$ converges to $u$ point-wise almost everywhere in $\Omega$ as $r\to 0$. This section is dedicated to the link between the local mean function and the $(p,\lambda)$--oscillation of the function $u$. Throughout this section, we assume that $\Omega$ is a $\theta$-smooth domain, and we fix $p\in[1,+\infty)$ and $\lambda\geq 0$.

\begin{lemma}\label{lem:1} For any $u\in L^p(\Omega)$ and $0<r_1\leq r_2< +\infty$ we have
\begin{equation}\nonumber
|u_{r_1}(x)-u_{r_2}(x)|\leq C_1\left(\frac{r_1^\lambda+r_2^\lambda}{r_1^d}\right)^{1/p}O_\Omega^{p,\lambda}[u](x),\quad \forall x\in\Omega. 
\end{equation} 
where $C_1:=C_1(d,\theta,p) := (2^{p-1}/(\theta v_d))^{1/p}$.
\end{lemma}

\begin{proof} By convexity, for any $x\in\Omega$,
\begin{equation}\nonumber
\begin{aligned}
|u_{r_1}(x)-u_{r_2}(x)|^p &\leq2^{p-1}\left(|u(z) -u_{r_1}(x)|^p+|u(z) -u_{r_2}(x)|^p\right),\quad\forall z\in\Omega,\\
\int_{B_\Omega(x,r_1)}\td z |u_{r_1}(x)-u_{r_2}(x)|^p &\leq2^{p-1}\left(\int_{B_\Omega(x,r_1)}|u(z) -u_{r_1}(x)|^p\td z+\int_{B_\Omega(x,r_1)}|u(z) -u_{r_2}(x)|^p\td z\right),\\
|B_\Omega(x,r_1)||u_{r_1}(x)-u_{r_2}(x)|^p &\leq2^{p-1}\left(\int_{B_\Omega(x,r_1)}|u(z) -u_{r_1}(x)|^p\td z+\int_{B_\Omega(x,r_2)}|u(z) -u_{r_2}(x)|^p\td z\right),\\
\theta v_dr_1^d|u_{r_1}(x)-u_{r_2}(x)|^p  &\leq 2^{p-1}\left(r_1^{\lambda}\left(O_\Omega^{p,\lambda}[u](x)\right)^p+r_2^{\lambda}\left(O_\Omega^{p,\lambda}[u](x)\right)^p\right),\\
&\leq 2^{p-1}\left(r_1^{\lambda}+r_2^{\lambda}\right)\left(O_\Omega^{p,\lambda}[u](x)\right)^p. 
\end{aligned}\end{equation}
We conclude by applying the power $1/p$.
\end{proof}

\begin{lemma}\label{lem:2} For any $u\in L^p(\Omega)$, $r>0$ and $n\in\N$, if $\lambda\neq d$ we have
\begin{equation}\nonumber
|u_{r}(x)-u_{\frac{r}{2^n}}(x)|\leq C_2r^{(\lambda-d)/p}\frac{1-\sigma^n}{1-\sigma}O_\Omega^{p,\lambda}[u](x),\quad \forall x\in\Omega,
\end{equation}
where $C_2=C_2(d,\theta,p,\lambda)$ and $\sigma:=2^{(d-\lambda)/p}$.
\end{lemma}

\begin{proof}Define $r_k=r/2^k$ and apply the previous Lemma, for any $x\in\Omega$,
\begin{equation}\nonumber\begin{aligned}
|u_{r_{k+1}}(x)-u_{r_{k}}(x)| &\leq Cr^{(\lambda-d)/p}\left(2^{(d-\lambda)/p}\right)^{k+1}(1+2^\lambda)^{1/p}O_\Omega^{p,\lambda}[u](x)\\
&\leq C_2r^{(\lambda-d)/p}\sigma^kO_\Omega^{p,\lambda}[u](x)\\
\end{aligned}\end{equation}
where $C_2=C_2(d,\theta,p,\lambda):=C_1(d,\theta,p)\left(2^{(d-\lambda)/p}\right)(1+2^\lambda)^{1/p}$ and $\sigma:=2^{(d-\lambda)/p}$.  Now by triangular inequality,
\begin{equation}\nonumber\begin{aligned}
|u_{r}(x)-u_{\frac{r}{2^n}}(x)| &\leq \sum_{k=0}^{n-1} |u_{r_{k+1}}(x)-u_{r_{k}}(x)|\leq  C_2r^{(\lambda-d)/p}O_\Omega^{p,\lambda}[u](x) \sum_{k=0}^{n-1}\sigma^k\\
&\leq C_2r^{(\lambda-d)/p}O_\Omega^{p,\lambda}[u](x)\frac{1-\sigma^n}{1-\sigma}. 
\end{aligned}\end{equation}
\end{proof}

\begin{proposition}\label{prop:36} For any $u\in L_C^{p,\lambda,q}(\Omega)$ with $\lambda>d$, the sequence $(u_r)_{r>0}$ converges to $u$ in $L^q(\Omega)$ when $r$ goes to zero. More precisely, we have
\begin{equation}\label{eq:ur}\begin{aligned}
|u_{r}(x)-u(x)|\leq C_2r^\alpha\frac{1}{1-\sigma}O_\Omega^{p,\lambda}[u](x),\quad\forall r>0,\quad \fae\ x\in\Omega,
\end{aligned}\end{equation}
where $\alpha:=(d-\lambda)/p$ and
\begin{equation}\nonumber
\norm{u_r-u}{L^q(\Omega)}\leq C_3r^{\alpha}\norm{u}{L_C^{p,d+\alpha p,q}(\Omega)},\quad \forall r>0,
\end{equation}
where $C_3=C_3(d,\theta,p,\lambda)>0$.
\end{proposition}

\begin{proof} Simply remark that in this case $\sigma <1$ and take the limit when $n$ goes to infinity, as $u_{r/2^n}$ converges almost everywhere to $u(x)$ via the Lebesgue differentiation theorem,
\begin{equation}\nonumber\begin{aligned}
|u_{r}(x)-u(x)|\leq C_2r^{\alpha}\frac{1}{1-\sigma}O_\Omega^{p,\lambda}[u](x),\quad\forall r>0,\quad \fae\ x\in\Omega.
\end{aligned}\end{equation}
Now take the $L^q$-norm and remember that $\norms{O_\Omega^{p,\lambda}[u]}{L^q(\Omega)}\leq \norm{u}{L_C^{p,\lambda,q}(\Omega)}$ to conclude. 
\end{proof}

\begin{proposition}\label{prop:bound}For any $u\in L_C^{p,\lambda,q}(\Omega)$ with $\lambda<d$ the local mean $u_r$ satisfies
\begin{equation}\nonumber
|u_r(x)|\leq \frac{1}{|\Omega|^{1/p}}\norm{u}{L^p(\Omega)} + C_4r^{-\alpha}O_\Omega^{p,\lambda}[u](x),\quad  \forall r>0,\quad \forall x\in\Omega.
\end{equation}
where $\alpha:=(d-\lambda)/p$ and $C_4:=\frac{C_2}{2^{\alpha}-1}+C_1(1+2^{\lambda})^{1/p}$. 
\end{proposition}

\begin{proof} Call $r_0:=\diam(\Omega)$, fix $r\leq r_0$ and pick $n\in\N$ such that $\frac{r_0}{2^{n+1}}\leq r\leq \frac{r_0}{2^{n}}$ and write 
\begin{equation}\nonumber
|u_r(x)|\leq |u_{r_0}(x)| + |u_{r_0}(x) - u_{\frac{r_0}{2^{n}}}(x)| +  |u_{\frac{r_0}{2^{n}}}(x) - u_{r}(x)|.
\end{equation} 
The first right hand term is bounded by $\frac{1}{|\Omega|^{1/p}}\norm{u}{L^p(\Omega)}$. The second right hand term is bounded using Lemma \ref{lem:2}:
\begin{equation}\nonumber
|u_{r_0}(x)-u_{\frac{r_0}{2^n}}(x)|\leq C_2r_0^{-\alpha}\frac{1-\sigma^n}{1-\sigma}O_\Omega^{p,{\lambda}}[u](x),\quad \forall x\in\Omega,\end{equation} 
with $\sigma=2^{\alpha}>1$. Remark that $(1-\sigma^n)/(1-\sigma)\leq \sigma^n/(\sigma-1)$ and $r_0^{-\alpha}\sigma^n=\left(\frac{r_0}{2^n}\right)^{-\alpha}\leq r^{-\alpha}$ so
\begin{equation}\nonumber
|u_{r_0}(x)-u_{\frac{r_0}{2^n}}(x)|\leq \frac{C_2}{\sigma-1}r^{-\alpha}O_\Omega^{p,{\lambda}}[u](x),\quad \forall x\in\Omega.\end{equation} 
The third right hand term is bounded using Lemma \ref{lem:1}:
\begin{equation}\nonumber
|u_{\frac{r_0}{2^n}}(x)-u_{r}(x)|\leq C_1\left(\frac{r^{\lambda}+\left(\frac{r_0}{2^n}\right)^{\lambda}}{r^d}\right)^{1/p}O_\Omega^{p,{\lambda}}[u](x),\quad \forall x\in\Omega,
\end{equation} 
and using that $\frac{r_0}{2^n}\leq 2r$ we get that
\begin{equation}\nonumber
|u_{\frac{r_0}{2^n}}(x)-u_{r}(x)|\leq C_1(1+2^{\lambda})^{1/p}r^{-\alpha}O_\Omega^{p,{\lambda}}[u](x),\quad \forall x\in\Omega.
\end{equation} 
Putting all things together it comes that
\begin{equation}\nonumber
|u_r(x)|\leq \frac{1}{|\Omega|^{1/p}}\norm{u}{L^p(\Omega)} + C_4r^{-\alpha}O_\Omega^{p,{\lambda}}[u](x),\quad \forall x\in\Omega.
\end{equation}
This estimation is still valid if $r>r_0$ as in this case, $|u_r(x)|\leq\norm{u}{L^p(\Omega)}|\Omega|^{-1/p}$. 
\end{proof}

\begin{lemma}\label{lem:38} For any $u\in L_C^{1,d+\alpha,q}(\Omega)$, $x,y\in\Omega$, $r=2|x-y|$, we have

\begin{equation}\nonumber
|u_r(x)-u_r(y)|\leq C_5r^\alpha\left(O_\Omega^{1,d+\alpha}[u](x) + O_\Omega^{1,d+\alpha}[u](y)\right).
\end{equation}
where $C_5=C_5(d,\theta):=2^d/(\theta v_d)$.
\end{lemma}

\begin{proof} For any $z\in\Omega$,
\begin{equation}\nonumber\begin{aligned}
|u_r(x)-u_r(y)| &\leq |u(z)-u_r(x)|+|u(z)-u_r(y)|,\\
|B_\Omega(x,r)\cap B_\Omega(y,r)||u_r(x)-u_r(y)| &\leq \int_{B_\Omega(x,r)\cap B_\Omega(y,r)}|u(z)-u_r(x)|\td z\\ &+\int_{B_\Omega(x,r)\cap B_\Omega(y,r)}|u(z)-u_r(y)|\td z,\\
&\leq \norm{u-u_r(x)}{L^1(B_\Omega(x,r))} + \norm{u-u_r(y)}{L^1(B_\Omega(y,r))}.
\end{aligned}\end{equation}
Considering that $B_\Omega(x,r/2)\subset B_\Omega(x,r)\cap B_\Omega(y,r)$ we have that 
\begin{equation}\nonumber
|B_\Omega(x,r)\cap B_\Omega(y,r)|\geq \frac{\theta v_d}{2^d}r^d,
\end{equation}
we deduce that
\begin{equation}\nonumber\begin{aligned}
|u_r(x)-u_r(y)| &\leq \frac{2^d}{\theta v_d}r^{\alpha}\left(O_\Omega^{1,d+\alpha}[u](x) + O_\Omega^{1,d+\alpha}[u](y)\right).
\end{aligned}\end{equation}
\end{proof}

\begin{corollary}\label{cor:39} For any $u\in L_C^{1,d+\alpha,q}(\Omega)$ with $\alpha>0$ we have
\begin{equation}\nonumber
|u(x)-u(y)| \leq C_6\left(O_\Omega^{1,d+\alpha}[u](x)+O_\Omega^{1,d+\alpha}[u](y)\right)|x-y|^\alpha,\quad\fae\ (x,y)\in\Omega^2,
\end{equation}
where $C_6:=C_6(d,\theta,\alpha)$.
\end{corollary}

\begin{proof} Take $u\in L_C^{1,d+\alpha,q}(\Omega)$ and $x,y\in \Omega$ and denote $r=2|x-y|$. Write
\begin{equation}\nonumber
|u(x)-u(y)|\leq |u(x)-u_r(x)| + |u(y)-u_r(y)| + |u_r(x)-u_r(y)|.
\end{equation}
Using the argument of the proof of Proposition \ref{prop:36}, the first term is bounded by 
\begin{equation}\nonumber\begin{aligned}
|u_{r}(x)-u(x)|\leq C_3r^{\alpha}O_\Omega^{1,d+\alpha}[u](x),\quad \fae  x\in\Omega,
\end{aligned}\end{equation}
and second term is bounded by 
\begin{equation}\nonumber\begin{aligned}
|u_{r}(y)-u(y)|\leq C_3r^{\alpha}O_\Omega^{1,d+\alpha}[u](y),\quad \fae  y\in\Omega.
\end{aligned}\end{equation}
The third term is bounded thanks to Lemma \ref{lem:38}
\begin{equation}\nonumber\begin{aligned}
|u_r(x)-u_r(y)|\leq C_5r^{\alpha}\left(O_\Omega^{1,d+\alpha}[u](x) + O_\Omega^{1,d+\alpha}[u](y)\right),\quad \fae(x,y)\in\Omega^2. 
\end{aligned}\end{equation}
Just add the three terms to conclude.
\end{proof}

\section{Decay estimates for harmonic functions}\label{app:decay}

In this appendix, we collect several classical estimates for harmonic functions on balls that are key tools in the derivation of regularity estimates in generalized Morrey--Campanato spaces. We call for $r>0$ the ball $B_r:=B(0,r)\subset\R^d$, we also denote $\nabla^\ell u$ the tensor of all the $\ell$--order partial derivative of a function $u:\R^d\to\R^p$. 

\begin{proposition}[Interior estimates for harmonic functions] Let $\ell\in\N$, there exists a constant $C$ depending only on $\ell$ and $d$ such that any harmonic function $u$ on $B_1$ satisfies
\[
\|\nabla^\ell u\|_{L^2(B_{1/2})}
\le
C\|u\|_{L^2(B_1)}.
\]
\end{proposition}

\begin{proof}
Let
\[
r_j=\frac12+\frac{\ell-j}{2\ell},
\qquad j=0,\ldots,\ell.
\]
Then
\[
r_0=1,\qquad r_\ell=\frac12,
\qquad r_j-r_{j+1}=\frac1{2\ell}.
\]
Since derivatives of harmonic functions are harmonic, the famous Caccioppoli
inequality gives for every $j\in\{0,\ldots,\ell-1\}$
\[
\|\nabla^{j+1}u\|_{L^2(B_{r_{j+1}})}
\leq
\frac{C}{r_j-r_{j+1}}
\|\nabla^j u\|_{L^2(B_{r_j})},
\]
where $C$ is an universal constant. Iterating these estimates, we obtain
\begin{equation}\nonumber
\begin{aligned}
\|\nabla^\ell u\|_{L^2(B_{1/2})}
&\leq
C^\ell
\prod_{j=0}^{\ell-1}
(r_j-r_{j+1})^{-1}
\|u\|_{L^2(B_1)}\\
&=
(2\ell C)^\ell
\|u\|_{L^2(B_1)}.
\end{aligned}
\end{equation}
This concludes the proof.
\end{proof}

%

\begin{corollary} Let $k,\ell\in\N$, there exists a constant $C$ depending only on $\ell$ and $d$ such that any harmonic function $u$ on $B_1$ satisfies
\begin{equation}\nonumber
\int_{B_{1/2}}|\nabla^{k+\ell} u|^2\leq C\int_{B_1}|\nabla^ku-(\nabla^ku)_{B_1}|^2.
\end{equation}
Moreover, for any $n\in\N$ there exists a constant $C>0$ depending only on $n$ and $d$ such that
\begin{equation}\nonumber\begin{aligned}
\norm{\nabla u}{H^n(B_{1/2})}^2&\leq C\int_{B_1}|u-u_{B_1}|^2,\\
\norm{\nabla^2u}{H^n(B_{1/2})}^2&\leq C\int_{B_1}|\nabla u-(\nabla u)_{B_1}|^2.
\end{aligned}\end{equation}
\end{corollary}

\begin{proof} Let $w:=\nabla^k u-(\nabla^k u)_{B_1}$. Each component of $w$ is harmonic on $B_1$. Using the previous result, and since $\nabla^\ell w=\nabla^{k+\ell}u$ we obtain
\[
\int_{B_{1/2}}|\nabla^{k+\ell}u|^2 = \int_{B_{1/2}}|\nabla^\ell w|^2
\leq
C(\ell,d)
\int_{B_1}|w|^2
\leq 
C(\ell,d)
\int_{B_1}|\nabla^k u-(\nabla^k u)_{B_1}|^2.
\]
This proves the first estimate. Taking \(k=0\) and summing over \(\ell=1,\dots,n+1\), we get
\[
\|\nabla u\|_{H^n(B_{1/2})}^2
\le
\wt C(n,d)
\int_{B_1}|u-u_{B_1}|^2.
\]
Taking \(k=1\) and summing over \(\ell=1,\dots,n+1\), we obtain
\[
\|\nabla^2u\|_{H^n(B_{1/2})}^2
\le
\wt C(n,d)
\int_{B_1}|\nabla u-(\nabla u)_{B_1}|^2.
\]
\end{proof}

\begin{corollary}[Pointwise bounds] There exist two positive constants $C_1$ and $C_2$ depending only on $d$ such that any harmonic function $u$ on $B_1$ satisfies 
\begin{equation}\nonumber\begin{aligned}
\sup_{B_{1/2}}|\nabla u|^2&\leq C_1\int_{B_1}|u - u_{B_1}|^2\leq C_2\int_{B_1}|\nabla u|^2,\\
\sup_{B_{1/2}}|\nabla^2 u|^2&\leq C_1\int_{B_1}|\nabla u-(\nabla u)_{B_1}|^2 \leq C_2\int_{B_1}|\nabla^2 u|^2.
\end{aligned}\end{equation}
\end{corollary}

\begin{proof} The two first inequalities (involving $C_1$) come from the Sobolev embedding $H^n(B_{1/2})\inj L^\infty(B_{1/2})$ for any integer $n>d/2$, and the previous corollary. The second inequalities (involving $C_2$) come from Poincaré inequality in $B_1$.
\end{proof}

\begin{theorem}[Decay estimates for harmonic functions]\label{theo:decay} There exists a constant $C>0$ depending only on $d$ such that for any harmonic function $u$ in a ball $B_r:=B(x_0,r)$ and any $\rho\in(0,r]$ one has
\begin{equation}\nonumber
\int_{B_\rho}|\nabla u|^2 \leq C\frac{\rho^{d}}{r^d}\int_{B_r}|\nabla u|^2,\\
\end{equation}
and 
\begin{equation}\nonumber\begin{aligned}
\int_{B_\rho}|\nabla u-(\nabla u)_{B_\rho}|^2 \leq C\frac{\rho^{d+2}}{r^{d+2}}\int_{B_r}|\nabla u-(\nabla u)_{B_r}|^2.
\end{aligned}\end{equation}
\end{theorem}
\begin{proof} In order to use the previous corollary, we define $U(y):=u(x_0+ry)$ which is harmonic in $B_1$. For any $\rho\leq r/2$,
\begin{equation}\nonumber\begin{aligned}
\int_{B_\rho}|\nabla u|^2=r^{d-2}\int_{B_{\rho/r}}|\nabla U|^2 &\leq r^{d-2}|B_{\rho/r}|\sup_{B_{\rho/r}} |\nabla U|^2 \\
&\leq  r^{d-2}|B_{\rho/r}|\sup_{B_{1/2}} |\nabla U|^2 \leq  C\frac{\rho^{d}}{r^d}r^{d-2}\int_{B_1}|\nabla U|^2\\
&\leq C\frac{\rho^{d}}{r^d}\int_{B_r}|\nabla u|^2.
\end{aligned}\end{equation}
For $\rho\in(r/2,r]$, this inequality is trivial if one chooses $C\geq 2^d$. For the second inequality note that $(\nabla u)_{B_\rho}=\frac{1}{r}(\nabla U)_{B_{\rho/r}}$ and so for any $\rho\leq r/2$, using Poincaré in $B_\rho$ 
\begin{equation}\nonumber\begin{aligned}
\int_{B_\rho}|\nabla u-(\nabla u)_{B_\rho}|^2&\leq C\rho^2\int_{B_\rho}|\nabla^2 u|^2\leq C\rho^2r^{d-4}\int_{B_{\rho/r}}|\nabla^2 U|^2\\
&\leq C\rho^2r^{d-4}|B_{\rho/r}|\sup_{B_{1/2}}|\nabla^2 U|^2\leq C\frac{\rho^{d+2}}{r^{d+2}}r^{d-2}\int_{B_1}|\nabla U-(\nabla U)_{B_1}|^2\\
&\leq C\frac{\rho^{d+2}}{r^{d+2}}\int_{B_r}|\nabla u-(\nabla u)_{B_r}|^2.
\end{aligned}\end{equation}
Once again, for $\rho\in(r/2,r]$, this inequality is trivial as soon as $C\geq 2^{d+2}$.
\end{proof}

\section{The iteration lemma}\label{app:iteration}

\begin{lemma}\label{lem:it} Let $\ph:(0,r_0]\to \R$ be a nonnegative  and nondecreasing function which satisfies

\begin{equation}\nonumber
\ph(\rho)\leq A\left(\frac{\rho^\alpha}{r^\alpha} + Br^{\beta}\right)\ph(r) + Cr^\gamma,\ \ 0<\rho\leq r\leq r_0,
\end{equation}
for some positive constants $A,\alpha,\beta,\gamma$ with $\gamma<\alpha$ and some nonnegative constants $B,C$. 
 Then $r\mapsto r^{-\gamma}\ph(r)$ is bounded on $(0,r_0]$ and there exists two positive constants $D_1:=D_1(A,\alpha,\gamma)$ and $D_2:=D_2(A,\alpha,\beta,\gamma)$ such that 
   \begin{equation}\nonumber
r^{-\gamma}\ph(r)\leq D_1\left(r_0^{-\gamma}\ph(r_0) +D_2B^{\gamma/\beta}\ph(r_0)+ C\right),\quad \forall r\in(0,r_0].
 \end{equation}
\end{lemma}

\begin{proof} Choose first $\delta = (\gamma+\alpha)/2$. Consider
\begin{equation}\nonumber
\tau=\min\left(\left(\frac{1}{2A}\right)^{1/(\alpha-\delta)},\frac{1}{2}\right).
\end{equation}
It satisfies $A\tau^\alpha\leq\tau^\delta/2$. Consider also 
\begin{equation}\nonumber
r:=
\left\{\begin{aligned}
&r_0\quad\text{if}\ B=0,\\
&\min\left(\left(\frac{\tau^\delta}{2AB}\right)^{1/\beta},r_0\right)\quad\text{otherwise}.
\end{aligned}\right.
\end{equation}
It satisfies $ABr^\beta\leq\tau^\delta/2$. For any $\rho\in(0,r]$, we have
 \begin{equation}\nonumber\begin{aligned}
 \ph(\tau\rho) \leq A(\tau^\alpha+B\rho^\beta)\ph(\rho) + C\rho^\gamma\leq\tau^\delta\ph(\rho) + C\rho^\gamma.
 \end{aligned}\end{equation}
 Call $u_n:=\ph(\tau^n r)$ for any $n\in\N$, it satisfies 
 \begin{equation}\nonumber
 u_{n+1}\leq \tau^\delta u_n + Cr^\gamma\tau^{n\gamma},\quad \forall n\in\N,
 \end{equation}
 and by induction
 \begin{equation}\nonumber
 u_n\leq \ph(r)\tau^{n\delta} + Cr^\gamma\tau^{n\delta}\sum_{k=0}^{n-1}\tau^{(\gamma-\delta)k},\quad\forall n\in\N^*.
 \end{equation}
Using that $\tau^{\gamma-\delta}>1$, the second right hand side term is bounded by

\begin{equation}\nonumber
\tau^{n\delta}\sum_{k=0}^{n-1}\tau^{(\gamma-\delta)k}=\tau^{n\delta}\frac{\tau^{(\gamma-\delta)n} - 1}{\tau^{(\gamma-\delta)} - 1}
= \tau^{n\gamma}\frac{1-\tau^{(\delta-\gamma)n}}{\tau^{(\gamma-\delta)} - 1}\leq \tau^{n\gamma}\frac{1}{\tau^{(\gamma-\delta)} - 1},
\end{equation}
and then
\begin{equation}\nonumber
\ph(\tau^nr)\leq \ph(r)\tau^{n\delta} + Cr^\gamma\tau^{n\gamma}\frac{1}{\tau^{(\gamma-\delta)} - 1},\quad\forall n\in\N.
\end{equation}
 
  Now for any $\rho\in(0,r]$, there exists $n\in\N$ such that $\tau^{n+1}r<\rho\leq\tau^n r$. Using the fact that $\ph$ is nondecreasing, we get that 
   \begin{equation}\nonumber
\ph(\rho)\leq \ph(r)\tau^{n\delta} + Cr^\gamma\tau^{n\gamma}\frac{1}{\tau^{(\gamma-\delta)} - 1},
 \end{equation}
  and as $\tau^n<\frac{\rho}{\tau r}$ we end up to
 \begin{equation}\nonumber
\ph(\rho)\leq \tau^{-\delta}\left(\frac{\rho}{r}\right)^\delta\ph(r) + C\rho^{\gamma}\frac{\tau^{-\gamma}}{\tau^{(\gamma-\delta)} - 1}.
 \end{equation}
 Call now $D_1:=D_1(A,\alpha, \gamma):=\max\left(1,\tau^{-\delta},\frac{\tau^{-\gamma}}{\tau^{(\gamma-\delta)} - 1}\right)$ to rewrite 
\begin{equation}\nonumber
\ph(\rho)\leq D_1\left(\left(\frac{\rho}{r}\right)^\delta\ph(r) + C\rho^{\gamma}\right),\quad 0<\rho\leq r.
 \end{equation}
 As $\rho\leq r$, this implies that $\rho\mapsto \rho^{-\gamma}\ph(\rho)$ is bounded on $(0,r]$ by
 \begin{equation}\nonumber
\rho^{-\gamma}\ph(\rho)\leq D_1\left(r^{-\gamma}\ph(r_0) + C\right),\quad 0<\rho\leq r.
 \end{equation}
 Now for any $\rho\in(r,r_0]$, $\rho^{-\gamma}\ph(\rho)\leq r^{-\gamma}\ph(r_0)\leq D_1\left(r^{-\gamma}\ph(r_0) + C\right)$.
 It remains to bound $r^{-\gamma}$. Remember that if $B>0$, $r:=\min\left(\left(\frac{\tau^\delta}{2AB}\right)^{1/\beta},r_0\right)$ and so
 \begin{equation}\nonumber
 r^{-\gamma}\leq r_0^{-\gamma}+\left(\frac{2AB}{\tau^\delta}\right)^{\gamma/\beta}.
 \end{equation}
 This is still true if $B=0$. Call finally $D_2:=D_2(A,\alpha,\beta,\gamma):=\left(\frac{2A}{\tau^\delta}\right)^{\gamma/\beta}$ we write the final inequality 
  \begin{equation}\nonumber
\rho^{-\gamma}\ph(\rho)\leq D_1\left(r_0^{-\gamma}\ph(r_0) +D_2B^{\gamma/\beta}\ph(r_0)+ C\right),\quad 0<\rho\leq r_0.
 \end{equation}
\end{proof}

\section{Proof of Proposition \ref{prop:1}}\label{app:1}

\begin{proof} Direct: If $a$ satisfies $(H_2)$ we fix $r>0$ and for almost every $x\in \Omega$ we have
\begin{equation}\nonumber
\begin{aligned}
|a(x)-a(y)| &\leq g(x)r^\alpha,\quad \fae\ y\in B_\Omega(x,r)
\end{aligned}
\end{equation}
and then 
\begin{equation}\nonumber
\begin{aligned}
\norm{a-a(x)}{L^\infty(B_\Omega(x,r))} &\leq g(x)r^\alpha,\quad \fae\ x\in \Omega.
\end{aligned}
\end{equation}
Fix $x\in\Omega$ such that the above control is satisfied and write
\begin{equation}\nonumber
\begin{aligned}
\norm{a-a_r(x)}{L^\infty(B_\Omega(x,r))} &\leq \norm{a-a(x)}{L^\infty(B_\Omega(x,r))}+  |a_r(x)-a(x)| \\
&\leq \norm{a-a(x)}{L^\infty(B_\Omega(x,r))} + \fint_{B_\Omega(x,r)}|a(y)-a(x)|\td y\\
&\leq 2\norm{a-a(x)}{L^\infty(B_\Omega(x,r))}\\
&\leq 2g(x)r^\alpha. 
\end{aligned}
\end{equation}
From that, we deduce that $\omega_\alpha[a]\leq 2g$ almost everywhere in $\Omega$ and in particular $\omega_\alpha[a]\in L^q(\Omega)$.

Reverse: Assume that $\omega_\alpha[a](x)\in L^q(\Omega)$. By the Lebesgue differentiation theorem, we know that for almost every $x\in\Omega$, $a_r(x)\to a(x)$ when $r\to 0$. For those $x$, and for any $\rho\in(0,r)$ we write
\begin{equation}\nonumber\begin{aligned}
|a_\rho(x)-a_r(x)| &\leq |a(z)-a_\rho(x)|+|a(z)-a_r(x)|,\quad \forall z\in B_\Omega(x,\rho),\\
&\leq \norm{a-a_\rho(x)}{L^\infty(B_\Omega(x,\rho))} +  \norm{a-a_r(x)}{L^\infty(B_\Omega(x,\rho))},\\
&\leq \norm{a-a_\rho(x)}{L^\infty(B_\Omega(x,\rho))} +  \norm{a-a_r(x)}{L^\infty(B_\Omega(x,r))},\\
&\leq (\rho^\alpha+r^\alpha)\omega_\alpha[a](x).
\end{aligned}\end{equation}
Now take the limit $\rho\to 0$ to get  
\begin{equation}\nonumber
|a(x)-a_r(x)| \leq r^\alpha \omega_\alpha[a](x),\quad \forall r>0.
\end{equation}
Now, fix $r>0$, for almost every $y\in B_\Omega(x,r)$ such that $|x-y|\geq r/2$, we have 
\begin{equation}\nonumber
|a(y)-a_r(x)| \leq r^\alpha \omega_\alpha[a](x)\leq 2^\alpha \omega_\alpha[a](x)|x-y|^\alpha, 
\end{equation}
then for all $r>0$, almost every $x\in\Omega$ and  almost every $y\in B_\Omega(x,r)\bs B_\Omega(x,r/2)$, we have
\begin{equation}\nonumber
\begin{aligned}
|a(x)-a(y)| &\leq |a(x)-a_r(x)| + |a(y)-a_r(x)|\\
&\leq (2^{\alpha}+2) \omega_\alpha[a](x)|x-y|^\alpha.
\end{aligned}
\end{equation}
Denote $E_r:=\set{(x,y)\in\Omega}{r/2\leq |x-y| <r}$. The above inequality is true for every $r>0$ and for almost every $(x,y)\in E_r$. Hence, it is true for almost every $(x,y)$ in the countable union
\begin{equation}\nonumber
\bigcup_{r\in (0,+\infty)\bigcap \Q}E_r = \Omega^2\bs\{x=y\}. 
\end{equation}
Finally, 
\begin{equation}\nonumber
|a(x)-a(y)| \leq g(x)|x-y|^\alpha,\quad \fae\ (x,y)\in\Omega^2.
\end{equation}
with $g:=(2^{\alpha}+2) \omega_\alpha[a]$. 
\end{proof}

\end{appendix}

\bibliographystyle{plain}
\bibliography{biblio}

\begin{thebibliography}{10}

\bibitem{baison2017beltrami}
Antonio Bais{\'o}n, Albert Clop, and Joan Orobitg.
\newblock Beltrami equations with coefficient in the fractional sobolev space
  $w^{\theta,\frac 2\theta }$.
\newblock {\em Proceedings of the American Mathematical Society},
  145(1):139--149, 2017.

\bibitem{baison2017fractional}
Antonio~L Bais{\'o}n, Albert Clop, Raffaella Giova, Joan Orobitg, and Antonia
  Passarelli~di Napoli.
\newblock Fractional differentiability for solutions of nonlinear elliptic
  equations.
\newblock {\em Potential Analysis}, 46(3):403--430, 2017.

\bibitem{clop2009beltrami}
Albert Clop, Daniel Faraco, Joan Mateu, Joan Orobitg, and Xiao Zhong.
\newblock Beltrami equations with coefficient in the sobolev space w\^{}1,p.
\newblock 2009.

\bibitem{cruz2013beltrami}
Victor Cruz, Joan Mateu, and Joan Orobitg.
\newblock Beltrami equation with coefficient in sobolev and besov spaces.
\newblock {\em Canadian Journal of Mathematics}, 65(6):1217--1235, 2013.

\bibitem{de1957sulla}
Memoria di~Ennio De~Giorgi.
\newblock Sulla differenziabilitae l’analiticita delle estremali degli
  integrali multipli regolari.
\newblock {\em Ennio De Giorgi}, 167, 1957.

\bibitem{di2014higher}
A~Passarelli Di~Napoli.
\newblock Higher differentiability of minimizers of variational integrals with
  sobolev coefficients.
\newblock {\em Adv. Calc. Var}, 7(1):59--89, 2014.

\bibitem{di2014higher2}
A~Passarelli Di~Napoli.
\newblock Higher differentiability of solutions of elliptic systems with
  sobolev coefficients: the case p= n= 2.
\newblock {\em Potential Anal}, 41(3):715--735, 2014.

\bibitem{fefferman1972h}
Charles Fefferman and Elias~M Stein.
\newblock H\^{}p spaces of several variables.
\newblock 1972.

\bibitem{giaquinta2013introduction}
Mariano Giaquinta and Luca Martinazzi.
\newblock {\em An introduction to the regularity theory for elliptic systems,
  harmonic maps and minimal graphs}.
\newblock Springer Science \& Business Media, 2013.

\bibitem{gilbarg1998elliptic}
David Gilbarg, Neil~S Trudinger, David Gilbarg, and NS~Trudinger.
\newblock {\em Elliptic partial differential equations of second order},
  volume~2.
\newblock Springer, 1998.

\bibitem{kristensen2010boundary}
Jan Kristensen and Giuseppe Mingione.
\newblock Boundary regularity in variational problems.
\newblock {\em Archive for rational mechanics and analysis}, 198(2):369--455,
  2010.

\bibitem{kuusi2012universal}
Tuomo Kuusi and Giuseppe Mingione.
\newblock Universal potential estimates.
\newblock {\em Journal of Functional Analysis}, 262(10):4205--4269, 2012.

\bibitem{lerner2010some}
Andrei~K Lerner.
\newblock Some remarks on the fefferman-stein inequality.
\newblock {\em Journal d'Analyse Math{\'e}matique}, 112(1):329--349, 2010.

\bibitem{meyers1963p}
Norman~G Meyers.
\newblock An $l^p$--estimate for the gradient of solutions of second order
  elliptic divergence equations.
\newblock {\em Annali della Scuola Normale Superiore di Pisa-Scienze Fisiche e
  Matematiche}, 17(3):189--206, 1963.

\bibitem{mingione2007calderon}
Giuseppe Mingione.
\newblock The calder{\'o}n-zygmund theory for elliptic problems with measure
  data.
\newblock {\em Annali della Scuola Normale Superiore di Pisa-Classe di
  Scienze}, 6(2):195--261, 2007.

\bibitem{moser1960new}
J{\"u}rgen Moser.
\newblock A new proof of de giorgi's theorem concerning the regularity problem
  for elliptic differential equations.
\newblock {\em Communications on Pure and Applied Mathematics}, 13(3):457--468,
  1960.

\bibitem{nash1958continuity}
John Nash.
\newblock Continuity of solutions of parabolic and elliptic equations.
\newblock {\em American Journal of Mathematics}, 80(4):931--954, 1958.

\bibitem{rafeiro2012morrey}
Humberto Rafeiro, Natasha Samko, and Stefan Samko.
\newblock Morrey-campanato spaces: an overview.
\newblock {\em Operator Theory, Pseudo-Differential Equations, and Mathematical
  Physics: The Vladimir Rabinovich Anniversary Volume}, pages 293--323, 2012.

\bibitem{stein1970singular}
Elias~M Stein.
\newblock {\em Singular integrals and differentiability properties of
  functions}.
\newblock Number~30. Princeton university press, 1970.

\bibitem{troianiello2013elliptic}
Giovanni~Maria Troianiello.
\newblock {\em Elliptic differential equations and obstacle problems}.
\newblock Springer Science \& Business Media, 2013.

\end{thebibliography}
\end{document}